\documentclass{svproc}
\usepackage{url}

\usepackage{xcolor}
\usepackage{amssymb}
\usepackage{amsmath}
\usepackage{mathtools}
\usepackage{mathrsfs}
\usepackage{cite}
\usepackage{accents} 
\usepackage{tikz}
\usepackage{pgfplots}
\pgfplotsset{
	table/search path={./image}
}
\pgfplotsset{compat=newest}
\usepackage{subcaption}
\usepackage{booktabs}
\usepackage{siunitx}
\newcommand{\matlab}{MATLAB\textsuperscript{\textregistered}}

\usepgfplotslibrary{groupplots}

\pgfplotsset{
        colormap={parula}{
            rgb255=(53,42,135)
            rgb255=(15,92,221)
            rgb255=(18,125,216)
            rgb255=(7,156,207)
            rgb255=(21,177,180)
            rgb255=(89,189,140)
            rgb255=(165,190,107)
            rgb255=(225,185,82)
            rgb255=(252,206,46)
            rgb255=(249,251,14)
        },
}

\allowdisplaybreaks 

\newtheorem{assumption}[theorem]{Assumption}

\newcommand{\N}{\ensuremath\mathbb{N}}
\newcommand{\Z}{\ensuremath\mathbb{Z}}

\newcommand{\R}{\ensuremath\mathbb{R}}

\newcommand{\dt}{\,\mathrm{d}t}
\newcommand{\dx}{\,\mathrm{d}x}

\DeclareMathOperator{\diag}{diag}
\DeclareMathOperator{\sech}{sech}

\newcommand{\state}{z}

\newcommand{\timeInt}{\ensuremath\mathbb{T}}
\newcommand{\finalTime}{T}

\newcommand{\param}{\mu}
\newcommand{\paramSet}{\mathbb{M}}

\newcommand{\ROMstate}{\alpha}

\newcommand{\ROMdim}{r}

\newcommand{\coeff}{\alpha}
\newcommand{\coeffVector}{\boldsymbol{\coeff}}
\newcommand{\mode}{\varphi}
\newcommand{\modeVector}{\boldsymbol{\mode}}
\renewcommand{\path}{p}
\newcommand{\pathVector}{\boldsymbol{\path}}
\newcommand{\transform}{\mathcal{T}}

\newcommand{\BS}{\mathscr{X}}
\newcommand{\BSS}{\mathscr{Y}}
\newcommand{\PS}{\mathscr{P}}
\newcommand{\penaltyFunctional}{\Lambda}

\newcommand{\penaltyTerm}{\lambda}

\newcommand{\Admissible}{\mathcal{A}}

\begin{document}
\mainmatter              
\title{Decomposition of flow data via gradient-based transport optimization}
\titlerunning{Decomposition of flow data via gradient-based transport optimization}  
%
\author{Felix Black\inst{1} \and Philipp Schulze\inst{1} \and Benjamin Unger \inst{2}}
\authorrunning{Felix Black et al.} 
%
\tocauthor{Felix Black, Philipp Schulze, Benjamin Unger}
\institute{Institute of Mathematics, Technische Universität Berlin, Berlin, Germany\\
\email{\{black,pschulze\}@math.tu-berlin.de}
\and
Stuttgart Center for Simulation Science, Universität Stuttgart, Stuttgart, Germany\\
\email{benjamin.unger@simtech.uni-stuttgart.de}
}

\maketitle              

\begin{abstract}
We study an optimization problem related to the approximation of given data by a linear combination of transformed modes. In the simplest case, the optimization problem reduces to a minimization problem well-studied in the context of proper orthogonal decomposition. Allowing transformed modes in the approximation renders this approach particularly useful to compress data with transported quantities, which are prevalent in many flow applications. We prove the existence of a solution to the infinite-dimensional optimization problem. Towards a numerical implementation, we compute the gradient of the cost functional and derive a suitable discretization in time and space. We demonstrate the theoretical findings with three challenging numerical examples.
\end{abstract}

\section{Introduction}
\label{sec:intro}

Projection-based model order reduction (MOR) typically relies on the fact that the solution manifold of a (parametrized) differential equation can be approximately embedded in a low-dimensional linear subspace. The best subspace of a given dimension, where best is understood as the minimal worst-case approximation error, is characterized by the Kolmogorov $n$-widths \cite{Kol36}.  In practice, the minimizing subspace for the $n$-widths is difficult to compute. Instead, one relies on the proper orthogonal decomposition (POD) \cite{GubV17}, which is typically combined with a greedy-search within the parameter domain, to get an approximate solution. In more detail, for given parameters $\param_\sigma\in\paramSet$ ($\sigma=1,\ldots,\ell$), associated data samples $\state(t,x;\param_\sigma)$ with time variable $t\in\timeInt\vcentcolon=[0,\finalTime]$, space variable $x\in\Omega\subseteq \R^d$, and desired dimension $\ROMdim\in\N$ of the low-dimensional subspace, POD determines orthonormal basis functions of a low-dimensional subspace solving the minimization problem
\begin{equation}
	\label{eqn:PODminimization}
	\left\{\begin{aligned}
		&\min \frac{1}{2} \sum_{\sigma=1}^\ell \int_0^{\finalTime} \bigg\| \state(t,x;\param_\sigma) - \sum_{i=1}^\ROMdim \ROMstate_i(t;\param_\sigma) \varphi_i(x) \bigg\|^2 \dt\\
		&\mathrm{with}~\ROMstate_i(t;\param_\sigma)
		\vcentcolon=\langle\state(t,x;\param_\sigma),\varphi_i(x)\rangle\quad\text{for } i=1,\ldots,\ROMdim,\,
		\sigma=1,\ldots,\ell,\\
		& \mathrm{s.t.}~\langle\varphi_i,\varphi_j\rangle =
		\delta_{ij}\quad\text{for } i,j=1,\ldots,\ROMdim.
	\end{aligned}\right.
\end{equation}
If the $n$-widths, respectively the Hankel singular values for linear dynamical systems \cite{UngG19}, decay fast, then one can expect to construct an effective reduced-order model (ROM) able to approximate the full dynamics with a small approximation error.  Although one can show exponential decay for a large class of problems \cite{MadPT02b}, it is well-known, see for instance \cite{CagMS19,GreU19}, that the decay of the $n$-widths for flow problems is typically slow, thus conspiring against MOR.

To remedy this issue prevalent in transport-dominated phenomena, several strategies have been proposed in the literature. We refer to \cite{OhlR13,TadPQ15,RimML18,CagMS19,NonBRM19,Peh20,Tad20,Wel20,KraSR21} to name just a few.

One promising approach, introduced in \cite{ReiSSM18,SchRM19,LeeC19} and formalized in \cite{BlaSU20,BlaSU21a}, is to replace the POD minimization problem with
\begin{equation}
	\label{eqn:sPODminimization}
	\min \frac{1}{2} \sum_{\sigma=1}^\ell \int_0^{\finalTime} \bigg\| \state(t,x;\param_\sigma) - \sum_{i=1}^\ROMdim
	\ROMstate_i(t;\param_\sigma) \varphi_i(x-\path_i(t;\param_\sigma)) \bigg\|^2 \dt,
\end{equation}
thus accounting explicitly for the transportation of quantities throughout the spatial domain. 
Consequently, the linear subspace in the Kolmogorov $n$-widths is replaced with a subspace able to adapt itself to the solution over time, hence rendering this a nonlinear approach. 
Note that in contrast to the POD minimization problem \eqref{eqn:PODminimization}, we do not require the modes to be orthogonal to each other.
This is due to the fact that in the setting of \eqref{eqn:sPODminimization}, we would need to require orthogonality of $\varphi_i(x-\path_i(t;\param_\sigma))$ and $\varphi_j(x-\path_j(t;\param_\sigma))$ for all $i\neq j$, $t\in [0,\finalTime]$, and $\sigma=1,\ldots,\ell$, which is in general not a reasonable assumption, cf.~\cite[Ex.~4.4]{BlaSU20} for an illustrative example.

In the past years, there have been some attempts of solving discretized versions of \eqref{eqn:sPODminimization} or related minimization problems.
In \cite{ReiSSM18}, the authors propose a heuristic iterative method for computing a decomposition of a given snapshot matrix by an approximation ansatz as in \eqref{eqn:sPODminimization}.
The numerical experiments indicate promising results, but it is not clear in which situations the proposed method actually determines an optimal solution.
Another heuristic has been recently proposed in \cite[sec.~5.2.1]{BlaSU21b} and it is applied to snapshot data of a wildland fire simulation.
This method is based on a decomposition of the snapshot matrix and involves a small number of singular value decompositions without requiring an iterative procedure.
The numerical results presented in \cite{BlaSU21b} demonstrate the effectiveness of this approach, but it is in general not optimal in the sense of the minimization problem \eqref{eqn:sPODminimization}.
In contrast, the method introduced in \cite{SchRM19} directly solves a fully discretized version of \eqref{eqn:sPODminimization} by determining optimal modes $\varphi$ and coefficients $\ROMstate$, but assumes the paths $\path$ to be given or determined in a pre-processing step.
A similar optimization problem has been proposed in \cite{Rei21} and aims at approximating the snapshot matrix by a sum of matrices representing different reference frames while achieving a fast singular value decay in each of the reference frames.
The corresponding cost function is shown to be an upper bound for a fully discretized version of \eqref{eqn:sPODminimization} and the cost functions coincide for the special case that only one reference frame is considered, i.e., if all ansatz functions in \eqref{eqn:sPODminimization} are shifted by the same amount.
Again, the paths are not considered as part of the optimization problem, but instead determined in a pre-processing step via peak or front tracking.
On the contrary, the authors in \cite{MenBALK20} focus on determining optimal paths, whereas the determination of optimal ansatz functions and coefficients is not addressed.
Moreover, to simplify the optimization problem, the paths are sought within a low-dimensional subspace consisting of predefined time-dependent library functions.
As in the case of the other mentioned works, also the authors in \cite{MenBALK20} consider a fully discrete optimization problem. 

We conclude that a gradient-based algorithm for the full optimization problem~\eqref{eqn:sPODminimization} is currently not available. Besides, a rigorous proof showing that~\eqref{eqn:sPODminimization} has a solution is missing in the literature. A notable exception is provided in \cite[Thm.~4.6]{BlaSU20}, albeit under the assumption that the path variables $\path_i(t)$ are known a-priori. In this paper we aim to close this gap. Our main contributions are the following:
\begin{enumerate}
	\item We show in Theorem~\ref{thm:existenceMinimizingSolution} the existence of a minimizing solution for the optimization problem~\eqref{eqn:minimizationProblem}, which generalizes the minimization problem~\eqref{eqn:sPODminimization}.
	Afterward, we reformulate the constrained minimization problem~\eqref{eqn:minimizationProblem} as unconstrained problem~\eqref{eqn:minimizationProblem:unconstrained} by adding appropriate penalty terms and conclude from Theorem~\ref{thm:existenceMinimizingSolution} that also the unconstrained problem has a solution, cf.~Corollary~\ref{cor:existenceRegularizedMinimization}. In addition, Theorem~\ref{thm:convergencePenaltyMethod} details that the solution of the unconstrained problem converges to the solution of the constraint problem. 
	\item We compute the gradient of the unconstrained problem in Theorem~\ref{thm:gradient}, which enables the use of gradient-based methods to solve~\eqref{eqn:sPODminimization}.
		In this context, a remarkable finding is that the paths have to be sufficiently smooth (e.g. in $H^1(0,\finalTime)$), since otherwise some directional derivatives of the cost functional with respect to the paths may not exist, cf.~ Example~\ref{ex:counterExampleDirectionalDerivative}.
	\item We discuss the discretization of the gradient in space and time in section~\ref{sec:discretization} and explicitly compute the path-dependent inner products for the shift operator with periodic boundary conditions in Example~\ref{ex:innerProducts}.  
		Finally, the effectiveness of gradient-based optimization is demonstrated for several examples in section~\ref{sec:examples}.
\end{enumerate}

\paragraph*{Notation} 

We denote the space of real $m\times n$ matrices by $\R^{m\times n}$ and the transpose of a matrix $A$ is written as $A^\top$.
Furthermore, for a vector with $n$ entries all equal to one we use the symbol $\mathbf{1}_n$.
Besides, for abbreviating diagonal and blockdiagonal matrices we use 
\begin{equation*}
	\diag(a_1,\ldots,a_n) \vcentcolon=
	\begin{bmatrix}
		a_1 	& 				& \\
				& \ddots 	& \\
				&				& a_n
	\end{bmatrix}
	,\quad \mathrm{blkdiag}(A_1,\ldots, A_n) \vcentcolon=
	\begin{bmatrix}
		A_1 	& 				& \\
				& \ddots 	& \\
				&				& A_n
	\end{bmatrix}
	,
\end{equation*}
respectively, where $a_1,\ldots,a_n$ are scalars and $A_1,\ldots, A_n$ matrices of arbitrary size.
For the Kronecker product of two matrices $A$ and $B$ we write $A\otimes B$.
The space of square-integrable functions mapping from an interval $(a,b)$ to a Banach space $\BS$ is denoted by $L^2(a,b;\BS)$ and, similarly, the space of essentially bounded measurable functions by $L^\infty(a,b;\BS)$.
Furthermore, we use $H^1(a,b;\BS)$ for the Sobolev subspace of functions in $L^2(a,b;\BS)$ possessing also a weak derivative in $L^2(a,b;\BS)$.
The corresponding subspace consisting of $H^1(a,b;\BS)$ functions whose values at the boundaries $a$ and $b$ coincide is denoted by $H^1_{\mathrm{per}}(a,b;\BS)$.
Besides, for the space of continuous functions from $[a,b]$ to $\BS$ we use the symbol $C([a,b];\BS)$. 
For the special case $\BS=\R$, we omit the last argument, i.e., we write, for instance, $L^2(a,b)$ instead of $L^2(a,b;\R)$.

\section{Preliminaries and problem formulation}
\label{sec:prelim}

To formalize the optimization problem~\eqref{eqn:sPODminimization}, we introduce the following spaces and notation.
Consider a real Hilbert space $\left( \BS, \left< \cdot, \cdot \right>_{\BS} \right)$ with induced norm
$\|\cdot\|_{\BS}$, and let $\BSS$ denote a dense subspace of $\BS$ that itself is a reflexive Banach space with norm
$\|\cdot\|_{\BSS}$. Our standing assumption is that we are minimizing the mean-squared distance to the data $\state\in
L^2(0,\finalTime;\BSS)$ in the Bochner space $L^2(0,\finalTime;\BS)$ with the additional requirement that the modes are elements of $\BSS$.  

To formalize the meaning of $\varphi_i(x-\path_i(t))$ in~\eqref{eqn:sPODminimization}, we follow the notation in \cite{BlaSU20} and introduce a family of linear and bounded operators $\transform_i\colon \PS_i\to\mathscr{B} (\BS)$ with real, finite-dimensional vector space $\PS_i$, for which we postulate the following properties, taken from \cite[Ass.~4.1]{BlaSU20}.

\begin{assumption}
	\label{ass:Transformation}
	For every $i=1,\ldots,\ROMdim$, every $\mode_i \in \BSS$,  and every $\path_i\in\PS_i$,  the operator $\transform_i(\path_i)$ is $\BSS$-invariant, i.e., $\transform_i(\path_i)\BSS\subseteq \BSS$, and the mapping
	\begin{equation*}
		\PS_i\to\BS,\qquad \path_i\mapsto \transform_i(\path_i)\mode_i
	\end{equation*}
	is continuous.
\end{assumption}

A particular example for such a family of operators is given by the shift operator with periodic boundary conditions, see for instance \cite[Ex.~5.2]{BlaSU20}. For further examples we refer to \cite{BeyT04,KraSR21}.

For the ease of presentation, we restrict ourselves to the case $\PS_i=\R$, and emphasize that all results can be generalized to $\PS_i=\R^{m_i}$ for some $m_i\in\N$.  For 
\begin{equation}
	\label{eqn:combinedSpace}
	\mathscr{Z} \vcentcolon= L^{2}(0,\finalTime; \R^{\ROMdim}) \times H^{1}(0,\finalTime; \R^{\ROMdim}) \times \BSS^{\ROMdim}
\end{equation}
let us define the cost functional
\begin{equation}
	\label{eqn:costFunctional}
		J \colon \mathscr{Z} \to \R, \qquad
		 \left( \coeffVector, \pathVector, \modeVector \right) \mapsto \frac{1}{2} \left\| z - \sum_{i = 1}^{r} \coeff_{i}\transform_{i} \left( \path_{i} \right) \mode_{i} \right\|_{L^2(0,\finalTime;\BS)}^{2}
\end{equation}
and for $C>0$ the space
\begin{equation}
	\label{eqn:admissibleSet}
	\Admissible_C \vcentcolon= \left\{ (\coeffVector,\pathVector,\modeVector)\in\mathscr{Z} \,\left|\,
		\max \left\{ \| \mode_{i} \|_{\BSS}, \| \coeff_{i} \|_{L^{2} \left( 0,\finalTime \right)}, \| \path_{i} \|_{H^{1} \left( 0,\finalTime \right)} \right\} \leq C
	\right.\right\},
\end{equation}
where we use the notation $\coeffVector = (\coeff_1,\ldots,\coeff_\ROMdim)$ to denote the coefficients of $\coeffVector$ and analogously for $\pathVector$ and $\modeVector$. 
To ensure that the norm in~\eqref{eqn:costFunctional} is defined, we invoke the following assumption, which is for instance satisfied if the family of operators~$\transform_i(\cdot)$ is uniformly bounded, cf.~\cite[Lem~4.2]{BlaSU20}.

\begin{assumption}
	\label{ass:transformedModesInL2}
	For every $i=1,\ldots,\ROMdim$,  $\coeff_i\in L^2(0,\finalTime)$,  $\path_i\in H^1(0,\finalTime)$, and every $\mode_i\in\BSS$, we assume 
	\begin{displaymath}
		\coeff_i\transform_i(\path_i)\mode_i\in L^2(0,\finalTime;\BS).
	\end{displaymath}
\end{assumption}

With these preparations, the constrained minimization problem that we are interested in takes the form 
\begin{align}
	\label{eqn:minimizationProblem}
	\min_{\left( \coeffVector, \pathVector, \modeVector \right)} J \left( \coeffVector, \pathVector, \modeVector \right), \quad \text{s.\,t. } \left( \coeffVector, \pathVector, \modeVector \right) \in \Admissible_C.
\end{align}
Before we proceed with our main results, let us make the following remarks:
\begin{itemize}
	\item To simplify the notation, we have implicitly set $\ell=1$ in~\eqref{eqn:sPODminimization}, thus assuming a single data sample. We emphasize that it is straightforward to generalize all results to $\ell>1$.
	\item The restriction of the optimization parameters to the admissible set stated in \eqref{eqn:admissibleSet} is used for proving the existence of a minimizing solution, cf.~Theorem~\ref{thm:existenceMinimizingSolution}, and helps to circumvent the problem that the cost functional in \eqref{eqn:costFunctional} is not coercive.
		For instance, since the transformed modes $\transform_{i} \left( \path_{i} \right) \mode_{i}$ may be
		linearly dependent, the corresponding coefficients may become unbounded, even though the value of the cost functional remains constant.
		The restriction to an admissible set of bounded functions as in \eqref{eqn:admissibleSet} allows to avoid such difficulties.
		Moreover, also from an application point of view, we note that it is usually reasonable to restrict to
		bounded optimization parameters, since when decomposing flow data, for example, we are usually not interested in unbounded coefficients, discontinuous paths, or modes which are less regular than the given flow data.
\end{itemize}

\section{Main results}
\label{sec:mainResults}

As first main result, we discuss the existence of a solution for the optimization problem~\eqref{eqn:minimizationProblem}, thus generalizing \cite[Thm.~4.6]{BlaSU20} to include the path variables.

\begin{theorem}
	\label{thm:existenceMinimizingSolution}
	Assume that the reflexive Banach space $\BSS$ is compactly embedded into $\BS$, and let $\state\in
	L^2(0,\finalTime;\BSS)$.  Furthermore, let the family of transformation operators satisfiy Assumptions~\ref{ass:Transformation} and~\ref{ass:transformedModesInL2}.  Then the constraint minimization problem~\eqref{eqn:minimizationProblem} has a solution for every $C>0$.
\end{theorem}
\begin{proof}
	The proof follows along the lines of the proof of \cite[Thm.~4.6]{BlaSU20}, with slight modifications to account for the optimization with respect to the path variables.  Let $C>0$. We first observe that the optimization problem possesses a finite infimum $J^\star\geq 0$. This follows directly from $J\geq 0$ and $(0,0,0)\in \Admissible_C$. We may thus choose a sequence $(\coeffVector^k,\pathVector^k,\modeVector^k)_{k\in\N}\in\Admissible_C$ satisfying
	\begin{displaymath}
		\lim_{k\to\infty} J(\coeffVector^k,\pathVector^k,\modeVector^k) = J^\star.
	\end{displaymath}
	Additionally, we have
	\begin{multline*}
		\left\| (\coeffVector^k, \pathVector^k, \modeVector^k) \right\|_{L^2(0,\finalTime; \R^\ROMdim) \times H^1(0,\finalTime; \R^\ROMdim) \times \BSS^\ROMdim}^2
		\\
		= \| \coeffVector^k \|_{L^2(0,\finalTime; \R^\ROMdim)}^2 + \|\pathVector^k \|_{H^1(0,\finalTime; \R^\ROMdim)}^2 + \| \modeVector^k \|_{\BSS^\ROMdim}^2
		\leq 3 r C^2
	\end{multline*}
	for all $k\in\N$, such that the Eberlein-\u{S}muljan theorem \cite[Thm.~21.D]{Zei90a} ensures the existence of a weakly convergent subsequence
	$( \coeffVector^{k_n}, \pathVector^{k_n}, \modeVector^{k_n})_{n \in \N}\subseteq \Admissible_C$ with weak limit $( \coeffVector^{\star}, \pathVector^{\star}, \modeVector^{\star})\in\Admissible_C$, cf.~\cite[Prop.~21.23\,(c)]{Zei90a}.
	Due to the compact embeddings $\BSS\hookrightarrow\BS$ and $H^1(0,\finalTime; \R^{\ROMdim}) \hookrightarrow L^2(0,\finalTime; \R^{\ROMdim})$, we conclude that $(\modeVector^{k_n})_{n\in\N}$ and $(\pathVector^{k_n})_{n\in\N}$ converge strongly in $\BS$ and $L^2(0,\finalTime; \R^{\ROMdim})$ to $\modeVector^\star$ and $\pathVector^\star$, respectively, cf.~\cite[Prop.~21.35]{Zei90a}. Using \cite[Thm.~3.12]{Rud86}, we conclude the existence of a subsequence, for which we use the same indexing, such that $(\pathVector^{k_n})_{n\in\N}$ converges pointwise to $\pathVector^\star$ for almost all $t\in(0,\finalTime)$.
	
	For the next part of the proof, we introduce the mapping
	\begin{align}
		\label{eqn:betaMapping}
		\beta \colon \mathscr{Z}\to L^2(0,\finalTime;\BS),\qquad
		(\coeffVector,\pathVector,\modeVector) \mapsto \sum_{i=1}^\ROMdim \coeff_i\transform_i(\path_i)\mode_i
	\end{align}
	with $\mathscr{Z}$ as defined in~\eqref{eqn:combinedSpace} and notice
	\begin{displaymath}
		J(\coeffVector,\pathVector,\modeVector) = \frac{1}{2} \|z - \beta(\coeffVector,\pathVector,\modeVector)\|_{L^2(0,\finalTime;\BS)}^2.
	\end{displaymath}
	If $\beta(\coeffVector^{k_n},\pathVector^{k_n},\modeVector^{k_n})$ converges weakly to $\beta(\coeffVector^\star,\pathVector^\star,\modeVector^\star)$, then the weak sequential lower semicontinuity of the norm, see for instance \cite[Prop.~21.23\,(c)]{Zei90a}, implies that $(\coeffVector^\star,\pathVector^\star,\modeVector^\star)$ is a minimizer of $J$. 
	It thus remains to show that $\beta(\coeffVector^{k_n},\pathVector^{k_n},\modeVector^{k_n})$ converges weakly to $\beta(\coeffVector^\star,\pathVector^\star,\modeVector^\star)$. 
	
	To this end,  we observe that
	\begin{align*}
		&\| \transform_{i} (\path_i^{k_n}(t)) \mode_i^{k_n} - \transform_{i} (\path_{i}^{\star}(t))\mode^{\star} \|_{\BS}\\
		&\leq \| \transform_{i} ( \path_{i}^{k_n}(t)) \mode_{i}^{k_n} - \transform_{i}(\path_{i}^{\star}(t)) \mode_{i}^{k_n} \|_{\BS}
		+ \| \transform_{i} (\path_{i}^{\star}(t)) \mode_{i}^{k_n} - \transform_{i} (\path_{i}^{\star}(t)) \mode_{i}^{\star} \|_{\BS}
	\end{align*}
	together with Assumption~\ref{ass:Transformation} and the strong convergence of $(\mode_i^{k_n})_{n\in\N}$ in $\BS$ implies
	\begin{align*}
		 \| \transform_{i} (\path_i^{k_n}(t)) \mode_i^{k_n} - \transform_{i} (\path_{i}^{\star}(t))\mode^{\star} \|_{\BS} \to 0\qquad \text{for } n\to\infty
	\end{align*}
	for $i=1,\ldots,\ROMdim$ and almost all $t\in(0,T)$.  Let $f\in L^2(0,\finalTime;\BS)$. Then clearly
	\begin{displaymath}
		\langle f(t), \transform_{i} (\path_i^{k_n}(t)) \mode_i^{k_n}\rangle_{\BS} \to \langle f(t), \transform_{i} (\path_i^\star(t)) \mode_i^\star\rangle_{\BS}\qquad\text{for } n\to\infty
	\end{displaymath}
	for $i=1,\ldots,\ROMdim$ and almost all $t\in(0,T)$ such that \cite[Prop~21.23\,(j)]{Zei90a} implies
	\begin{align*}
		\sum_{i=1}^\ROMdim \left\langle \coeff_i^{k_n}, \langle f, \transform_i(\path_i^{k_n})\mode_i^{k_n}\rangle_{\BS}\right\rangle_{L^2(0,\finalTime)} \to \sum_{i=1}^\ROMdim \left\langle \coeff_i^\star, \langle f, \transform_i(\path_i^\star)\mode_i^\star\rangle_{\BS}\right\rangle_{L^2(0,\finalTime)}
 	\end{align*}
 	for $n\to\infty$ 	and thus
 	\begin{align*}
 		\beta(\coeffVector^{k_n},\pathVector^{k_n},\modeVector^{k_n}) \rightharpoonup \beta(\coeffVector^\star,\pathVector^\star,\modeVector^\star)\qquad\text{for } n\to\infty,
 	\end{align*}
 	which completes the proof.
	\qed
\end{proof}

For numerical methods, it may be easier to work with unconstrained optimization problems. To this end, we use a penalty method, see for instance \cite[Cha.~13.1]{LueY16}, i.e., we add the constraint equation with a penalty parameter to the cost functional. In more detail, we assume for $C>0$ a penalty functional
\begin{equation}
	\label{eqn:abstractPenaltyFunction}
	\penaltyFunctional_C\colon \mathscr{Z}\to\R
\end{equation}
with the following properties available.

\begin{assumption}
	\label{ass:penaltyFunctional}
	The penalty functional~\eqref{eqn:abstractPenaltyFunction} is continuous, weakly sequentially lower semicontinuous, non-negative and has the following properties:
\begin{itemize}
	\item We have $\penaltyFunctional_C(\coeffVector,\pathVector,\modeVector) = 0$ if, and only if, $(\coeffVector,\pathVector,\modeVector)\in \Admissible_C$.
	\item For any sequence $(\coeffVector^k,\pathVector^k,\modeVector^k)$ with 
		\begin{displaymath}
			\max\{\|\coeffVector^k\|_{L^2(0,\finalTime;\R^{\ROMdim})}, \|\pathVector^k\|_{H^1(0,\finalTime;\R^{\ROMdim})}, \|\modeVector^k\|_{\BSS^\ROMdim}\} \to \infty\qquad\text{for } k\to\infty,
		\end{displaymath}
		we have $\penaltyFunctional_C(\coeffVector^k,\pathVector^k,\modeVector^k) \to \infty$ for $k\to\infty$.
\end{itemize}
\end{assumption}

\begin{example}
	The penalty functional
	\begin{align*}
		\penaltyFunctional_C(\coeffVector,\pathVector,\modeVector) \vcentcolon= \sum_{i=1}^\ROMdim \max\{0,\max\{\| \coeff_{i} \|_{L^{2}(0,\finalTime)},\left\Vert \path_{i} \right\Vert_{H^{1}(0,T)},\left\Vert \mode_{i} \right\Vert_{\BSS}\}-C\}
	\end{align*}
	satisfies Assumption~\ref{ass:penaltyFunctional}.
\end{example}

The penalized cost functional is then given as
\begin{displaymath}
	\widetilde{J}_C(\coeffVector,\pathVector,\modeVector,\penaltyTerm) \vcentcolon= J(\coeffVector,\pathVector,\modeVector) + \penaltyTerm \penaltyFunctional_C(\coeffVector,\pathVector,\modeVector)
\end{displaymath}
with penalty coefficient $\penaltyTerm>0$. The associated (unconstrained) minimization problem is thus given by
\begin{align}
	\label{eqn:minimizationProblem:unconstrained}
	\min_{(\coeffVector, \pathVector, \modeVector)\in\mathscr{Z}} \widetilde{J}_C( \coeffVector, \pathVector, \modeVector,\penaltyTerm)
\end{align}
with $\mathscr{Z}$ as defined in~\eqref{eqn:combinedSpace} and given $\penaltyTerm>0$.

\begin{corollary}
	\label{cor:existenceRegularizedMinimization}
	Let the assumptions of Theorem~\ref{thm:existenceMinimizingSolution} and Assumption~\ref{ass:penaltyFunctional} be satisfied. Then for any $C>0$ and any $\penaltyTerm>0$ the optimization problem~\eqref{eqn:minimizationProblem:unconstrained} has a solution. 
\end{corollary}

\begin{proof}
	Similarly as in the proof of Theorem~\ref{thm:existenceMinimizingSolution}, we conclude the existence of a finite infimum, such that we can choose a minimizing sequence $(\coeffVector^k,\pathVector^k,\modeVector^k)\in \mathscr{Z}$. Due to Assumption~\ref{ass:penaltyFunctional}, we deduce that $(\coeffVector^k,\pathVector^k,\modeVector^k)_{k\in\N}$ is bounded in $\mathscr{Z}$, i.e., there exists some $\widetilde{C}>0$ such that $(\coeffVector^k,\pathVector^k,\modeVector^k)\in \Admissible_{\widetilde{C}}$ for all $k\in\N$. The remaining proof thus follows along the lines of the proof of Theorem~\ref{thm:existenceMinimizingSolution}. \qed
\end{proof}

\begin{theorem}
	\label{thm:convergencePenaltyMethod}
	Let $(\lambda^k)_{k\in\N}\subseteq\R$ denote a non-decreasing sequence of positive numbers with
	$\lim_{k\to\infty} \lambda^k = \infty$, and let $C>0$.  For $k\in\N$, let
	$(\coeffVector^k,\pathVector^k,\modeVector^k)\in\mathscr{Z}$ denote a solution
	of~\eqref{eqn:minimizationProblem:unconstrained} with penalty parameter $\lambda^k$. If the assumptions of Corollary~\ref{cor:existenceRegularizedMinimization} are satisfied, then any limit point of $(\coeffVector^k,\pathVector^k,\modeVector^k)_{k\in\N}$ is a solution of~\eqref{eqn:minimizationProblem}.
\end{theorem}

\begin{proof}
	The proof follows along the lines of the proof of the main theorem in \cite[Cha.~13.1]{LueY16}.
	Let $(\coeffVector^\star,\pathVector^\star,\modeVector^\star)\in\Admissible_C$ denote a minimizer of~\eqref{eqn:minimizationProblem} with minimum $J^\star$. Then for every $k\in\N$ we have
	\begin{align*}
		\widetilde{J}_C(\coeffVector^k,\pathVector^k,\modeVector^k,\lambda^k) &\leq  \widetilde{J}_C(\coeffVector^{k+1},\pathVector^{k+1},\modeVector^{k+1},\lambda^k)\\
		&\leq \widetilde{J}_C(\coeffVector^{k+1},\pathVector^{k+1},\modeVector^{k+1},\lambda^{k+1})
	\end{align*}
	and 
	\begin{align*}
		J(\coeffVector^k,\pathVector^k,\modeVector^k) \leq \widetilde{J}_C(\coeffVector^k,\pathVector^k,\modeVector^k,\lambda^k) \leq \widetilde{J}_C(\coeffVector^\star,\pathVector^\star,\modeVector^\star,\lambda^k) = {J}(\coeffVector^\star,\pathVector^\star,\modeVector^\star) = J^\star.
	\end{align*}
	Thus $(\widetilde{J}_C(\coeffVector^k,\pathVector^k,\modeVector^k,\lambda^k))_{k\in\N}$ is a monotone sequence bounded above by $J^\star$. We thus set
	\begin{equation}
		\label{eqn:helper:a}
		\widetilde{J}_C^\star \vcentcolon= \lim_{k\to\infty} \widetilde{J}_C(\coeffVector^k,\pathVector^k,\modeVector^k,\lambda^k) \leq J^\star.
	\end{equation}
	Let $(\coeffVector^{k_n},\pathVector^{k_n},\modeVector^{k_n})$ denote a convergent subsequence with limit $(\coeffVector^\dagger,\pathVector^\dagger,\modeVector^\dagger)$ and set 
	\begin{equation}
		\label{eqn:helper:b}
		J^\dagger \vcentcolon= \lim_{n\to\infty} J(\coeffVector^{k_n},\pathVector^{k_n},\modeVector^{k_n}) = J(\coeffVector^\dagger,\pathVector^\dagger,\modeVector^\dagger), 
	\end{equation}
	using the continuity of $J$.
	Subtracting~\eqref{eqn:helper:a} from~\eqref{eqn:helper:b} yields
	\begin{displaymath}
		\lim_{n\to\infty} \lambda^{k_n} \penaltyFunctional_C(\coeffVector^{k_n},\pathVector^{k_n},\modeVector^{k_n}) = \widetilde{J}_C^\star - J^\dagger.
	\end{displaymath}
	Assumption~\ref{ass:penaltyFunctional} and $\lambda^{k_n}\to\infty$ for $n\to\infty$ together with the continuity of $\penaltyFunctional_C$ thus implies
	\begin{displaymath}
		\penaltyFunctional_C(\coeffVector^\dagger,\pathVector^\dagger,\modeVector^\dagger) = \lim_{n\to\infty} \penaltyFunctional_C(\coeffVector^{k_n},\pathVector^{k_n},\modeVector^{k_n}) = 0,
	\end{displaymath}
	showing $(\coeffVector^\dagger,\pathVector^\dagger,\modeVector^\dagger)\in\Admissible_C$. We conclude
	\begin{align*}
		J^\dagger = \lim_{n\to\infty} J(\coeffVector^{k_n},\pathVector^{k_n},\modeVector^{k_n}) \leq J^\star,
	\end{align*}
	which completes the proof. \qed
\end{proof}

Although~\eqref{eqn:minimizationProblem:unconstrained} is an unconstrained optimization problem, we still have to choose a suitable constant $C>0$ for the admissible set. 
Let us emphasize that the proofs of Theorem~\ref{thm:existenceMinimizingSolution} and Corollary~\ref{cor:existenceRegularizedMinimization} heavily depend on the fact that we have bounded sequences, which is the main reason for the constant $C>0$ in the admissible set~\eqref{eqn:admissibleSet}. 
However, we observed faster convergence in our numerical experiments when considering the unconstrained minimization problem without penalization.
For this reason and the sake of a concise presentation, we consider in the following only the unconstrained optimization problem~\eqref{eqn:minimizationProblem:unconstrained} with penalty parameter $\lambda=0$.
Nevertheless, we emphasize that adding the derivatives of the penalty terms to the gradient formulas is straightforward as long as the partial Fréchet derivatives of $\penaltyFunctional_C$ are available.

To solve the optimization problem~\eqref{eqn:minimizationProblem:unconstrained} with penalty parameter $\lambda=0$ numerically, we employ a gradient-based algorithm and thus have to compute the gradient of the objective function~\eqref{eqn:costFunctional}. 
It is easy to see that the directional derivatives of $J$ with respect to the coefficient function $\coeffVector\in L^2(0,\finalTime;\R^\ROMdim)$ and the modes $\modeVector\in\BSS^{\ROMdim}$ in directions $\boldsymbol{d} \in L^2( 0,\finalTime; \R^\ROMdim)$ and $\boldsymbol{h}\in \BSS^{\ROMdim}$, respectively, are given by
\begin{subequations}
	\label{eqn:directionalDerivative}
\begin{align}
	\label{eqn:directionalDerivative:coeff}	
	\partial_{\coeffVector,\boldsymbol{d}} J(\coeffVector,\pathVector,\modeVector) &= \sum_{i=1}^\ROMdim \left \langle \sum_{j=1}^\ROMdim \coeff_j\transform_j(\path_j)\mode_j-z,d_i\transform_i(\path_i)\mode_i\right\rangle_{\!\!\!L^2(0,\finalTime;\BS)},\\
	\label{eqn:directionalDerivative:mode}	\partial_{\modeVector,\boldsymbol{h}} J(\coeffVector,\pathVector,\modeVector) &= \sum_{i=1}^\ROMdim \left \langle \sum_{j=1}^\ROMdim \coeff_j\transform_j(\path_j)\mode_j-z,\coeff_i\transform_i(\path_i) h_i\right\rangle_{\!\!\!L^2(0,\finalTime;\BS)}.
\end{align}
The situation is slightly different for the partial derivative with respect to the path variable. First of all, we have to ensure that the transformed modes are differentiable (with respect to the path variable), i.e., we have to evoke the following assumption.
\begin{assumption}
	\label{ass:TransformationDerivative}
	For every $\mode_i \in \BSS$ and every $i=1,\ldots,\ROMdim$, the mapping
	\begin{align*}
		\R\to\BS,\qquad \path_i \mapsto \transform_i(\path_i)\mode_i,
	\end{align*}
	is continuously differentiable with derivatives in $\BS$.  For $\path_i\in\R$ we denote the derivative by $\frac{\partial}{\partial \path_i}\transform_{i}
	\left( \path_i \right) \mode_i\in\BS$ and assume $\coeff_i\frac{\partial}{\partial \path_i} \transform_i(\path_i)\mode_i\in L^2(0,\finalTime;\BS)$ for all $\coeff_i\in L^2(0,\finalTime)$ and all $\path_i\in H^1(0,\finalTime)$. 
\end{assumption}

In this case,  the directional derivative in direction $\boldsymbol{g}\in H^1(0,\finalTime;\R^{\ROMdim})$ is given as
\begin{equation}
	\label{eqn:directionalDerivative:path}
	\partial_{\pathVector,\boldsymbol{g}} J(\coeffVector,\pathVector,\modeVector) = \sum_{i=1}^\ROMdim \left \langle
	\sum_{j=1}^\ROMdim \coeff_j\transform_j(\path_j)\mode_j-z,\coeff_i \left[\tfrac{\partial}{\partial \path_{i}} \transform_i(\path_i)\mode_i\right] g_i\right\rangle_{\!\!\!L^2(0,\finalTime;\BS)}.
\end{equation}
\end{subequations}
Note that the Sobolev embedding theorems, see for instance \cite[Thm.~21.A.(d)]{Zei90a}, imply $g_i\in C([0,\finalTime]) \subseteq L^\infty(0,\finalTime)$, such that~\eqref{eqn:directionalDerivative:path} is defined.

\begin{theorem}
	\label{thm:gradient}
	Let the transformation operators satisfy Assumptions~\ref{ass:Transformation},~\ref{ass:transformedModesInL2}, and~\ref{ass:TransformationDerivative}.  Let $(\coeffVector,\modeVector,\pathVector)\in\mathscr{Z}$ and assume  
	\begin{subequations}
	\label{eqn:transformation:L2Linfty}
	\begin{align}
		\label{eqn:transformation:Linfty}
		\transform_i(\path_i)\mode_i &\in L^\infty(0,\finalTime;\BS),\\
		\label{eqn:transformation:operatorNormL2}\coeff_i\|\transform_i(\path_i)\| &\in L^2(0,\finalTime)
	\end{align}	
	\end{subequations}	
	for $i=1,\ldots,\ROMdim$, then the partial Fr\'echet derivatives of the cost functional $J$ (defined in~\eqref{eqn:costFunctional}) with respect to the coefficients, paths, and modes at $(\coeffVector,\pathVector,\modeVector)\in\mathscr{Z}$ are given by
	\begin{subequations}
	\begin{align}
		\label{eqn:frechet:coeff}
		\partial_{\coeffVector} J(\coeffVector,\pathVector,\modeVector)(\boldsymbol{d}) &\vcentcolon= \partial_{\coeffVector,\boldsymbol{d}} J(\coeffVector,\pathVector,\modeVector), &&\forall \boldsymbol{d}\in L^2(0,\finalTime;\R^\ROMdim),\\
		\label{eqn:frechet:path}\partial_{\pathVector} J(\coeffVector,\pathVector,\modeVector)(\boldsymbol{g}) &\vcentcolon= \partial_{\pathVector,\boldsymbol{g}} J(\coeffVector,\pathVector,\modeVector),&&\forall \boldsymbol{g}\in H^1(0,\finalTime;\R^{\ROMdim}),\\
		\label{eqn:frechet:mode}\partial_{\modeVector} J(\coeffVector,\pathVector,\modeVector)(\boldsymbol{h}) &\vcentcolon= \partial_{\modeVector,\boldsymbol{h}} J(\coeffVector,\pathVector,\modeVector),&&\forall \boldsymbol{h}\in \BSS^{\ROMdim},
	\end{align}
	\end{subequations}
	with directional derivatives as defined in~\eqref{eqn:directionalDerivative}.
\end{theorem}

\begin{proof}
	It suffices to show that $J$ is partially Fr\'echet differentiable with respect to the coefficients, paths, and modes. 
	Let $(\coeffVector,\pathVector,\modeVector),(\boldsymbol{d},\boldsymbol{g},\boldsymbol{h})\in\mathscr{Z}$.  Using~\eqref{eqn:transformation:Linfty} we obtain
	\begin{align*}
		J(\coeffVector&+\boldsymbol{d},\pathVector,\modeVector)-J(\coeffVector,\pathVector,\modeVector) - \partial_{\coeffVector,\boldsymbol{d}} J(\coeffVector,\pathVector,\modeVector)
		= \frac{1}{2} \left\|\sum_{i=1}^\ROMdim d_i\transform_i(\path_i)\mode_i\right\|_{L^2(0,\finalTime;\BS)}^2\\
		&\leq \frac{1}{2}\max_{i=1,\ldots,\ROMdim} \|\transform_i(\path_i)\mode_i\|_{L^\infty(0,\finalTime;\BS)}^2\left(\sum_{i=1}^\ROMdim \|d_i\|_{L^2(0,\finalTime)}\right)^2\\
		&\leq  \frac{\ROMdim^2}{2} \max_{i=1,\ldots,\ROMdim} \|\transform_i(\path_i)\mode_i\|_{L^\infty(0,\finalTime;\BS)}^2 \|\boldsymbol{d}\|_{L^2(0,\finalTime;\R^{\ROMdim})}^2
	\end{align*}
	and thus
	\begin{align*}
		\lim_{\|\boldsymbol{d}\|_{L^2(0,\finalTime;\R^{\ROMdim})}\to0} \frac{\left| J(\coeffVector+\boldsymbol{d},\pathVector,\modeVector)-J(\coeffVector,\pathVector,\modeVector) - \partial_{\coeffVector,\boldsymbol{d}} J(\coeffVector,\pathVector,\modeVector)\right|}{\|\boldsymbol{d}\|_{L^2(0,\finalTime;\R^{\ROMdim})}}  = 0.
	\end{align*}
	We conclude that $J$ is Fr\'echet differentiable with respect to the coefficients with Fr\'echet derivative as in~\eqref{eqn:frechet:coeff}. 
	For the partial derivative with respect to the modes we obtain
	\begin{align*}
		J(\coeffVector&,\pathVector,\modeVector+\boldsymbol{h})-J(\coeffVector,\pathVector,\modeVector) - \partial_{\modeVector,\boldsymbol{h}} J(\coeffVector,\pathVector,\modeVector)
		= \frac{1}{2} \left\|\sum_{i=1}^\ROMdim \coeff_i\transform_i(\path_i)h_i\right\|_{L^2(0,\finalTime;\BS)}^2\\
		&\leq \frac{1}{2} \int_0^\finalTime \left(\sum_{i=1}^\ROMdim |\coeff_i(t)|\|\transform_i(\path_i(t))\|\|h_i\|_{\BS}\right)^2 \dt\\
		&\leq \frac{\|\boldsymbol{h}\|_{\BS^{\ROMdim}}^2}{2} \int_0^\finalTime \left(\sum_{i=1}^\ROMdim |\coeff_i(t)|\|\transform_i(\path_i(t))\|\right)^2 \dt.
	\end{align*}	
	Using~\eqref{eqn:transformation:operatorNormL2}, we observe that the integral is finite. Similarly as before, we thus conclude that $J$ is Fr\'echet differentiable with respect to the modes with Fr\'echet derivative as in~\eqref{eqn:frechet:mode}.  We conclude our proof for the partial derivative with respect to the path variable.  
	Note that the Sobolev embedding theorem \cite[Thm~4.12, Part~I, Case A]{AdaF03} implies that the Sobolev space $H^1(0,\finalTime)$ is continuously embedded into the space $L^\infty(0,\finalTime)$, i.e., there exists a constant $\gamma>0$ independent of $g_i$, such that $\|g_i\|_{L^\infty(0,\finalTime)} \leq \gamma \|g_i\|_{H^1(0,\finalTime)} $.  
	We define 
	\begin{displaymath}
		f_i(\path_i,\mode_i,g_i) \vcentcolon= \transform_i(\path_i+g_i)\mode_i - \transform_i(\path_i)\mode_i - \left[\tfrac{\partial}{\partial \path_i} \transform_i(\path_i)\right]g_i.
	\end{displaymath}	
	For $g_i\equiv 0$ we have $f_i(\path_i,\mode_i,g_i) = 0$ for almost all $t\in(0,T)$. For $\|g_i\|_{H^1(0,\finalTime)}\neq 0$, let us define $\widehat{\mathbb{T}}_i \vcentcolon= \{t\in(0,T) \mid g_i(t)\neq 0\}$.
	Then
	\begin{multline*}
		\frac{\left\langle z,\coeff_i f_i(\path_i,\mode_i,g_i)\right \rangle_{L^2(0,\finalTime;\BS)}}{\|g_i\|_{H^1(0,\finalTime)}}  
		\leq \gamma \int_{\widehat{\mathbb{T}}_i} \coeff_i(t) \left\langle
		z(t),\frac{f_i(\path_i(t),\mode_i,g_i(t))}{|g_i(t)|} \right\rangle_{\BS} \dt.
	\end{multline*}
	From Assumption~\ref{ass:TransformationDerivative} we conclude
	\begin{displaymath}
		\lim_{\|g_i\|_{H^1(0,\finalTime)}\to0} \frac{ \left\langle z,\coeff_i f_i(\path_i,\mode_i,g_i)\right \rangle_{L^2(0,\finalTime;\BS)}}{\|g_i\|_{H^1(0,\finalTime)}} = 0,
	\end{displaymath}
	and thus 
	\begin{displaymath}
		\lim_{\|\boldsymbol{g}\|_{H^1(0,\finalTime;\R^\ROMdim)}\to0}	\frac{\sum_{i=1}^\ROMdim \left\langle z, \coeff_i f_i(\path_i,\mode_i,g_i)\right \rangle_{L^2(0,\finalTime;\BS)}}{\|\boldsymbol{g}\|_{H^1(0,\finalTime;\R^\ROMdim)}}  = 0.
	\end{displaymath}
	Furthermore, using $\beta$ as defined in~\eqref{eqn:betaMapping}, we obtain
	\begin{align*}
		\tfrac{1}{2}\|\beta(\coeffVector,\pathVector&+\boldsymbol{g},\modeVector)\|_{L^2(0,\finalTime;\BS)}^2 - \tfrac{1}{2}\left\|\beta(\coeffVector,\pathVector,\modeVector)\right\|_{L^2(0,\finalTime;\BS)}^2\\
		&\phantom{=}\qquad  - \sum_{j=1}^\ROMdim \left\langle \beta(\coeffVector,\pathVector,\modeVector), \coeff_j \left[\tfrac{\partial}{\partial \path_j} \transform_j(\path_j)\mode_j\right]g_j\right\rangle_{L^2(0,\finalTime;\BS)}\\
		&= \tfrac{1}{2}\left\|\beta(\coeffVector,\pathVector+\boldsymbol{g},\modeVector) - \beta(\coeffVector,\pathVector,\modeVector) + \beta(\coeffVector,\pathVector,\modeVector)\right\|_{L^2(0,\finalTime;\BS)}^2\\
		&\phantom{=}\qquad - \tfrac{1}{2}\left\|\beta(\coeffVector,\pathVector,\modeVector)\right\|_{L^2(0,\finalTime;\BS)}^2\\
		&\phantom{=}\qquad  - \sum_{j=1}^\ROMdim\left\langle \beta(\coeffVector,\pathVector,\modeVector), \coeff_j \left[\tfrac{\partial}{\partial \path_j} \transform_j(\path_j)\mode_j\right]g_j\right\rangle_{L^2(0,\finalTime;\BS)}\\
		&= \tfrac{1}{2} \|\beta(\coeffVector,\pathVector+\boldsymbol{g},\modeVector) - \beta(\coeffVector,\pathVector,\modeVector)\|_{L^2(0,\finalTime;\BS)}^2\\
		&\phantom{=}\qquad + \sum_{j=1}^\ROMdim\big\langle \beta(\coeffVector,\pathVector,\modeVector), \coeff_j f_j(\path_j,\mode_j,g_j)\big\rangle_{L^2(0,\finalTime;\BS)}.
	\end{align*}
	Similarly as before, we obtain
	\begin{gather*}
		\lim_{\|\boldsymbol{g}\|_{H^1(0,\finalTime;\R^\ROMdim)}\to0}	\frac{1}{2}\frac{\|\beta(\coeffVector,\pathVector+\boldsymbol{g},\modeVector) - \beta(\coeffVector,\pathVector,\modeVector)\|_{L^2(0,\finalTime;\BS)}^2}{\|\boldsymbol{g}\|_{H^1(0,\finalTime;\R^\ROMdim)}}  = 0,\\
		\lim_{\|\boldsymbol{g}\|_{H^1(0,\finalTime;\R^\ROMdim)}\to0}	\frac{\big\langle \beta(\coeffVector,\pathVector,\modeVector),\sum_{j=1}^\ROMdim \coeff_j f_j(\path_j,\mode_j,g_j)\big\rangle_{L^2(0,\finalTime;\BS)}}{\|\boldsymbol{g}\|_{H^1(0,\finalTime;\R^\ROMdim)}}   =0.
	\end{gather*}
	Combining the previous results, we thus infer 
	\begin{align*}
		\lim_{\|\boldsymbol{g}\|_{H^1(0,\finalTime;\R^\ROMdim)}\to0} \frac{|J(\coeffVector,\pathVector+\boldsymbol{g},\modeVector)-J(\coeffVector,\pathVector,\modeVector) - \partial_{\pathVector,\boldsymbol{g}} J(\coeffVector,\pathVector,\modeVector)|}{\|\boldsymbol{g}\|_{H^1(0,\finalTime;\R^\ROMdim)}}  = 0,
	\end{align*}
	which concludes the proof.
	\qed
\end{proof}

\begin{remark}
	If the family of transformation operators is uniformly bounded, i.e., there exists some $\overline{C}>0$ such that
	\begin{displaymath}
		\|\transform_i(\path_i)\|\leq \overline{C}\qquad\text{for all } \path_i\in\R,
	\end{displaymath}	
	then it is easy to see that condition~\eqref{eqn:transformation:L2Linfty} is satisfied. Note that in this case Assumption~\ref{ass:transformedModesInL2} is also satisfied, cf.~\cite[Lem~4.2]{BlaSU20}. An example for such a family of operators is (again) the periodic shift operator.
\end{remark}

Let us emphasize that it is essential for the directional derivative $\partial_{\pathVector,\boldsymbol{g}} J$ to have the path variable and associated directions in $H^1(0,\finalTime;\R^{\ROMdim})$. The following example details that if we take a direction in $L^2(0,\finalTime;\R^{\ROMdim})$, then the directional derivative may not be finite.

\begin{example}
	\label{ex:counterExampleDirectionalDerivative}
	Consider the shift operator $\transform(\path)\mode = \mode(\cdot-\path)$ with periodic embedding into the spaces $\BS \vcentcolon= L^2(0,2\pi)$ and $\BSS \vcentcolon= H^1_\mathrm{per}(0,2\pi)$, cf.~\cite[Ex.~4.3 and 5.12]{BlaSU20}.  
	It is well-known, that the shift operator is a semi-group with generator $-\tfrac{\partial}{\partial x} $, see for instance~\cite[Sec.~II.2.10]{EngN00}. 
	Let $z(t,x) = t^{-1/3}\cos(x)$, $\ROMdim=1$, $\path_1 \equiv 0$, and $\mode_1(x) = \sin(x)$.
	Then for any $\coeff_1,g_1\in L^2(0,\finalTime)$, we obtain
	\begin{align*}
		\partial_{\pathVector,\boldsymbol{g}} J(\coeff_1,\path_1,\mode_1) &= - \left\langle \coeff_1\mode_1-z,\coeff_1\tfrac{\partial}{\partial x}  \mode_1 g_1\right\rangle_{L^2(0,\finalTime;\BS)}\\
		&= \langle z,\coeff_1\tfrac{\partial}{\partial x} \mode_1g_1\rangle_{L^2(0,\finalTime;\BS)} = \|\cos\|^2_{\BS}\int_0^\finalTime t^{-1/3}\coeff_1(t)g_1(t) \dt.
	\end{align*}
	We notice that for $\coeff_1(t) = g_1(t) = t^{-1/3}$ we have $\coeff_1,g_1\in L^2(0,\finalTime)$ but the product $t^{-1/3}\coeff_1g_1$ is not in $L^1(0,\finalTime)$.  
	We conclude $\partial_{\pathVector,\boldsymbol{g}} J(\coeffVector,\pathVector,\modeVector)\not\in\R$.
\end{example}

\begin{remark}
	\label{rem:pathParameterization}
	To ensure $\pathVector\in H^1(0,\finalTime;\R^{\ROMdim})$ during a (numerical) optimization, we may choose a suitable low-dimensional subspace with continuously differentiable basis functions, such as the space of polynomials with given maximal degree.  The associated gradient is easily computed from Theorem~\ref{thm:gradient} via the chain rule. Besides the reduced computational cost, such an approach yields an interpretable representation for the wave speeds. We refer to \cite{MenBALK20} for a similar idea in a fully discretized setting.
\end{remark}

\section{Discretization}
\label{sec:discretization}

Towards a numerical implementation, we derive discretized versions of the partial derivatives from Theorem~\ref{thm:gradient}. To shorten notation, we introduce for $(\coeffVector,\pathVector,\modeVector)\in\mathscr{Z}$ and $\boldsymbol{h}\in\BSS^{\ROMdim}$ the quantities
\begin{subequations}
	\label{eqn:innerProductQuantities}
	\begin{align}
		v_i(\coeffVector,\pathVector,\modeVector) &\vcentcolon= \bigg\langle \sum_{j = 1}^\ROMdim \coeff_j \transform_j(\path_j)\mode_j-z, \transform_{i}(\path_i)\mode_i\bigg\rangle_{\!\!\!\BS},\\
		\xi_i(\coeffVector,\pathVector,\modeVector) &\vcentcolon= \bigg\langle \sum_{j = 1}^\ROMdim \coeff_j \transform_j(\path_j)\mode_j - z, \coeff_i\left[ \tfrac{\partial}{\partial \path_i} \transform_i(\path_i)\mode_i \right]\bigg\rangle_{\!\!\!\BS},\\
		\mu_i(\coeffVector,\pathVector, \modeVector, \boldsymbol{h}) &\vcentcolon= \bigg\langle \sum_{j =
		1}^\ROMdim \coeff_j\transform_j(\path_j)\mode_j -
		z,\coeff_i\transform_i(\path_i)h_i\bigg\rangle_{\!\!\!\BS},
	\end{align}
\end{subequations}
for $i=1,\ldots,\ROMdim$.

We start our exposition with the discretization with respect to time. To this end,  consider a time grid $0 =t_0 < t_1 < \ldots < t_m = \finalTime$
and associated quadrature rule defined by weights $\omega_\ell\geq 0$ for $\ell=0,\ldots,m$. The approximation of the directional derivative of $J$ with respect to $\coeffVector$ is thus given by
\begin{align*}
	\partial_{\coeffVector,\boldsymbol{d}} J (\coeffVector,\pathVector,\modeVector) 
	&= \sum_{i = 1}^\ROMdim \int_0^\finalTime v_i(\coeffVector(t),\pathVector(t),\modeVector) d_i(t) \dt\\
	&\approx \sum_{i = 1}^\ROMdim \sum_{k = 0}^m w_k v_i(\coeffVector(t_k),\pathVector(t_k),\modeVector) d_i(t_{k}) \\
	&= (\boldsymbol{v}^m(\coeffVector,\pathVector,\modeVector))^\top	\left( I_\ROMdim \otimes W \right) \boldsymbol{d}^m,
\end{align*}
where $\otimes$ is the Kronecker product and 
\begin{align*}
	W &\vcentcolon= \mathrm{diag} \left( w_{0}, \ldots, w_{m} \right) \in \R^{(m+1) \times (m+1)},\\
	\boldsymbol{v}_i^m(\coeffVector,\pathVector,\modeVector) &\vcentcolon=  \begin{bmatrix}
		v_i(\coeffVector(t_0),\pathVector(t_0),\modeVector) & \cdots v_i(\coeffVector(t_m),\pathVector(t_m),\modeVector)
	\end{bmatrix}^\top \in \R^{m+1},\\
	\boldsymbol{v}^m(\coeffVector,\pathVector,\modeVector) &\vcentcolon= \begin{bmatrix}
		\boldsymbol{v}_1^m(\coeffVector,\pathVector,\modeVector)^\top & \cdots &
		\boldsymbol{v}_\ROMdim^m(\coeffVector,\pathVector,\modeVector)^{\top}
	\end{bmatrix}^\top \in \R^{\ROMdim(m+1)},\\
	\boldsymbol{d}_i^m &\vcentcolon=  
	\begin{bmatrix}
		d_i(t_0) & \cdots d_i(t_m)
	\end{bmatrix}
	^\top \in \R^{m+1},\\
	\boldsymbol{d}^m &\vcentcolon= \begin{bmatrix}
		(\boldsymbol{d}_1^m)^{\top} & \cdots & (\boldsymbol{d}_\ROMdim^m)^{\top}
	\end{bmatrix}^\top \in \R^{\ROMdim (m+1)}.
\end{align*}
The time-discrete approximation of the partial derivative is thus given by
\begin{equation*}
	\label{eqn:variationCoefficientTimeDiscrete}
	\partial_{\coeffVector} J^m(\coeffVector,\pathVector,\modeVector) \vcentcolon= (\boldsymbol{v}^m(\coeffVector,\pathVector,\modeVector))^\top \left( I_\ROMdim \otimes W \right) \in \R^{1\times \ROMdim(m+1)}.
\end{equation*}
Analogously, the time-discrete approximation of the partial derivative of $J$ with respect to the path variables, i.e., $\partial_{\pathVector} J(\coeffVector,\pathVector,\modeVector)$, is given by
\begin{equation*}
	\label{eqn:variationPathTimeDiscrete}
	\partial_{\pathVector} J^m(\coeffVector,\pathVector,\modeVector) \vcentcolon=
	(\boldsymbol{\xi}^m(\coeffVector,\pathVector,\modeVector))^\top \left( I_\ROMdim \otimes W \right) \in
	\R^{1\times \ROMdim(m+1)},
\end{equation*}
with  $\boldsymbol{\xi}^m(\coeffVector,\pathVector,\modeVector)$ defined analogously as $\boldsymbol{v}^m(\coeffVector,\pathVector,\modeVector)$.
In the same fashion, we obtain the time-discrete approximation for the directional derivative with respect to the mode variables as
\begin{align*}
	\label{eqn:variationModeTimeDiscrete}
	\partial_{\modeVector,\boldsymbol{h}} J^m(\coeffVector,\pathVector,\modeVector) \vcentcolon=  (\boldsymbol{\mu}^m(\coeffVector,\pathVector,\modeVector,\boldsymbol{h}))^\top \left( I_\ROMdim \otimes W \right) \mathbf{1}_{\ROMdim(m+1)}\in \R,
\end{align*}
where we denote by $\mathbf{1}_{r(m+1)} \in \R^{r(m+1)}$ the vector with all entries equal to $1$, and $\boldsymbol{\mu}^m$ defined analogously as $\boldsymbol{v}^m$. 

For the spatial discretization, let $\BSS_n$ denote an $n$-dimensional subspace of $\BSS$ with basis functions $\psi_1,\ldots,\psi_n\in\BSS$. Let us define for $i,j=1,\ldots,\ROMdim$ and $\path_i,\path_j\in\R$ the matrices $M_{i,j}, N_{i,j},F_{i},G_{i}\in \R^{n \times n}$ via
\begin{subequations}
	\label{eqn:pathDependentInnerProducts}
	\begin{align}
		\left[ M_{i,j} (\path_i,\path_j)\right]_{k,\ell} &\vcentcolon= \langle \transform_j(\path_j) \psi_k, \transform_i(\path_i)\psi_\ell \rangle_{\BS}, \\
		\left[ N_{i,j}(\path_i,\path_j) \right]_{k,\ell} &\vcentcolon= \langle \transform_j(\path_j) \psi_k, \tfrac{\partial}{\partial \path_i}\transform_i(\path_i)\psi_\ell \rangle_{\BS}, \\
		\left[ F_{i}(\path_i) \right]_{k,\ell} &\vcentcolon= \langle \psi_k, \transform_i(\path_i) \psi_\ell \rangle_{\BS},\\
		\left[ G_{i}(\path_i) \right]_{k,\ell} &\vcentcolon= \langle \psi_k, \tfrac{\partial}{\partial \path_i}
		\transform_i(\path_i)\psi_{\ell} \rangle_{\BS},
	\end{align}
\end{subequations}
for $k,\ell=1,\ldots,\ROMdim$. 
For the data $z\in L^2(0,\finalTime;\BSS)$, the modes $\mode_i\in\BSS$, and directions $h_i\in\BSS$, we consider the approximations
\begin{displaymath}
	z(t) \approx \sum_{\ell=1}^n \widehat{z}_\ell(t)\psi_\ell,\qquad
	\mode_i \approx \sum_{\ell=1}^n \widehat{\mode}_{i,\ell} \psi_\ell,\qquad
	h_i \approx \sum_{\ell=1}^n \widehat{h}_{i,\ell} \psi_\ell,
\end{displaymath}
with
\begin{gather*}
	\boldsymbol{\widehat{z}}(t) \vcentcolon= [\widehat{z}_1(t)\; \cdots\; \widehat{z}_n(t)]^\top\in\R^n, \qquad\qquad
	\boldsymbol{\widehat{\mode}}_i \vcentcolon= [\widehat{\mode}_{i,1}\; \cdots\; \widehat{\mode}_{i,n}]^{\top}\in\R^n,\\
	\boldsymbol{\widehat{h}}_i \vcentcolon= [\widehat{h}_{i,1}\; \cdots\; \widehat{h}_{i,n}]^{\top}\in\R^n.
\end{gather*}
With these preparations, we obtain the spatial approximation of the inner products in~\eqref{eqn:innerProductQuantities} as
\begin{subequations}
\begin{align*}
	\widehat{v}_i(\coeffVector,\pathVector,\modeVector) &\vcentcolon= \bigg(\sum_{j=1}^{\ROMdim} \coeff_j \boldsymbol{\widehat{\mode}}_j^\top M_{i,j}(\path_i,\path_j) - \widehat{\boldsymbol{z}}^\top F_i(\path_i) \bigg)\boldsymbol{\widehat{\mode}}_i\in\R,\\
	\widehat{\xi}_i(\coeffVector,\pathVector,\modeVector) &\vcentcolon= \coeff_i\bigg(\sum_{j=1}^{\ROMdim} \boldsymbol{\widehat{\mode}}_j^\top N_{i,j}(\path_i,\path_j) - \widehat{\boldsymbol{z}}^\top G_i(\path_i) \bigg)\boldsymbol{\widehat{\mode}}_i\in\R,\\
	\widehat{\mu}_i(\coeffVector,\pathVector,\modeVector) &\vcentcolon= \coeff_i\bigg(\sum_{j=1}^{\ROMdim}
	\coeff_j \boldsymbol{\widehat{\mode}}_j^\top M_{i,j}(\path_i,\path_j) - \widehat{\boldsymbol{z}}^\top
	F_i(\path_i) \bigg)\in\R^{1 \times n}.
\end{align*}
\end{subequations}
We thus obtain the space- and time-discretized partial derivatives as
\begin{align*}
	\partial_{\coeffVector} \widehat{J}^m(\coeffVector,\pathVector,\modeVector) &\vcentcolon= (\widehat{\boldsymbol{v}}^m(\coeffVector,\pathVector,\modeVector))^\top \left( I_\ROMdim \otimes W \right) \in \R^{1\times (m+1)\ROMdim},\\
	\partial_{\pathVector} \widehat{J}^m(\coeffVector,\pathVector,\modeVector) &\vcentcolon= (\widehat{\boldsymbol{\xi}}^m(\coeffVector,\pathVector,\modeVector))^\top \left( I_\ROMdim \otimes W \right) \in \R^{1\times (m+1) \ROMdim},\\
	\partial_{\modeVector} \widehat{J}^m(\coeffVector,\pathVector,\modeVector) &\vcentcolon=  \mathbf{1}_{\ROMdim(m+1)}^\top \left( I_\ROMdim \otimes W \right)\mathcal{M}\in \R^{1\times n\ROMdim},
\end{align*}
with $\mathcal{M} \vcentcolon= \mathrm{blkdiag}(\widehat{\mu}_1^m(\coeffVector,\pathVector,\modeVector),\ldots,\widehat{\mu}_\ROMdim^m(\coeffVector,\pathVector,\modeVector))$.

We conclude this section with a specific computation of the quantities depending on the inner products for the periodic shift operator and $P_1$ finite elements.

\begin{example}
	\label{ex:innerProducts}
	Let us assume we have a one-dimensional domain $\Omega=(0,1)$ and a corresponding equidistant grid of step size $h\vcentcolon=\tfrac{1}{n}$.
	We discretize $\BSS = H^1_{\mathrm{per}}(\Omega)$ via periodic $P_1$ finite element functions.
	For $\BS = L^2(0,1)$ and shift operator with periodic embedding, we observe 
	\begin{equation*}
		M_{i,j}(\path_i,\path_j) = F_i(\path_i-\path_j)\qquad \text{and}\qquad N_{i,j}(\path_i,\path_j) = G_i(\path_i-\path_j).
	\end{equation*}
	For $\path_i = qh + \tilde{\path}_i$ with $q\in\Z$ and $\tilde{\path}_i\in[0,h)$ we obtain 
	\begin{align*}
		\resizebox{.99\linewidth}{!}{$
		\langle \psi_k,\transform_i(\path_i)\psi_\ell\rangle = \begin{cases}
			\tfrac{1}{h^2}\left(\tfrac{2}{3}(h-\tilde{p}_i)^3+\tilde{\path}_i(h-\tilde{\path}_i)^2+\tilde{\path}_ih(h-\tilde{\path}_i)+\tfrac16\tilde{\path}_i^3\right), & \text{if } \ell = k-q,\\
			 \tfrac{1}{6h^2}(h-\tilde{\path}_i)^3, & \text{if } \ell = k-q+1,\\
			 \tfrac{1}{h^2}\left(\tfrac{1}{6}(h-\tilde{\path}_i)^3 - \tfrac{1}{3}\tilde{\path}_i^3 + h^2\tilde{\path}_i\right),  & \text{if } \ell = k-q-1,\\			 
			 \tfrac{1}{6h^2} \tilde{\path}_i^3, & \text{if } \ell = k-q-2,\\
			 0, & \text{otherwise}.
		\end{cases}$}
	\end{align*}
	For further details, including the computation of $\langle \psi_k,\tfrac{\partial}{\partial
	\path_i}\transform_i(\path_i)\psi_\ell\rangle_\BS$, we refer to Appendix~\ref{sec:shiftedInnerProducts}.
\end{example}

\section{Numerical examples}
\label{sec:examples}

For our numerical examples, we use an equidistant time grid $t_i \vcentcolon= i\tau $ with step size $\tau>0$. 
The weights for the time integration are chosen based on the trapezoidal rule.
For the discretization in space, we also use an equidistant grid and follow Example~\ref{ex:innerProducts} for approximating the inner products occurring in the cost functional and the gradient.
The optimization itself is carried out with the \matlab package GRANSO with default settings, see \cite{CurMO17}.
The GRANSO algorithm is based on a quasi-Newton solver and can handle non-convex, non-smooth, constrained optimization problems.
For computing the relative $L^2$ errors of the approximations, we discretized the corresponding integrals by the trapezoidal rule.

For notational convenience, we assumed so far that there is exactly one mode per transformation operator. In practice, it is often more reasonable to cluster the modes into different reference frames, see, for instance, \cite[sec.~7.1]{BlaSU20}.
Thus, we use the clustered approximation ansatz
\begin{equation}
	\label{eqn:clusteredApproximation}
	z \approx \sum_{\rho=1}^q \sum_{i = 1}^{r_\rho}  \coeff_{\rho,i} \transform_{\rho} \left( \path_{\rho} \right) \mode_{\rho,i}
\end{equation}
for the following numerical experiments and emphasize that this only requires a minor and straightforward modification of the gradient.
We denote the approximation based on our optimization results with sPOD, not to be confused with the original sPOD algorithm from \cite{ReiSSM18}. 
Furthermore, we use dashed lines in the plots to display the (optimized) path variables.

\subsection{Viscous Burgers' equation}\label{sec:Burgers}

We consider the one-dimensional viscous Burgers' equation
\begin{align}
	\label{eqn:Burgers}
	\tfrac{\partial}{\partial t} z(t,x) &= \tfrac{1}{\mathrm{Re}} \tfrac{\partial^2}{\partial x^2} z(t,x) - z(t,x) \tfrac{\partial}{\partial x} z(t,x),  & (t,x)\in (0,2)\times(0,1),
\end{align}
and, following \cite{MauLB21}, use the analytical solution
\begin{align*}
	z(t,x) = \frac{x}{t+1} \bigg(1 + \sqrt{\tfrac{t+1}{\exp \left( \tfrac{\mathrm{Re}}{8} \right)}} \exp \left( \mathrm{Re}\tfrac{ x^2}{4t+4}\right)\!\!\!\bigg)^{\!\!-1}
\end{align*}
with Reynolds number  $\mathrm{Re} = 1000$ for our experiment.
We test our algorithm with data obtained from the analytical solution, evaluated on a grid with $100$ equidistant intervals in space and time, respectively.
We optimize for an approximation with a single frame, i.e., $q=1$
in~\eqref{eqn:clusteredApproximation}, and $\ROMdim_1 = \ROMdim = 2$ modes, 
and supply the first snapshot and the zero vector as starting values for the modes.
We initialize the corresponding coefficients as a constant function with value $1$. For the path, we start with a straight line given by $p(t) = \frac{37}{200} t$, evaluated at the time grid
points. The results are depicted in Figure~\ref{fig:result_burgers}, detailing that already with $\ROMdim=2$, an accurate approximation with a relative $L^2$ error of less than $\si{3\percent}$ can be achieved, while the POD approximation is not able to reproduce the shock front.
\begin{figure}
	\centering
	\begin{subfigure}[t]{.32\linewidth}
		\centering
		\begin{tikzpicture}
			\pgfplotsset{width=5cm,height=5cm}
			\begin{groupplot}[group style={group size=1 by 1, horizontal sep=2.5cm}]
				\nextgroupplot[enlargelimits=false,axis on top=false,xlabel=$x$,ylabel=$t$,ylabel style={rotate=-90},xtick=\empty,ytick=\empty]
				\addplot graphics [xmin=0,xmax=500,ymin=0,ymax=1000]{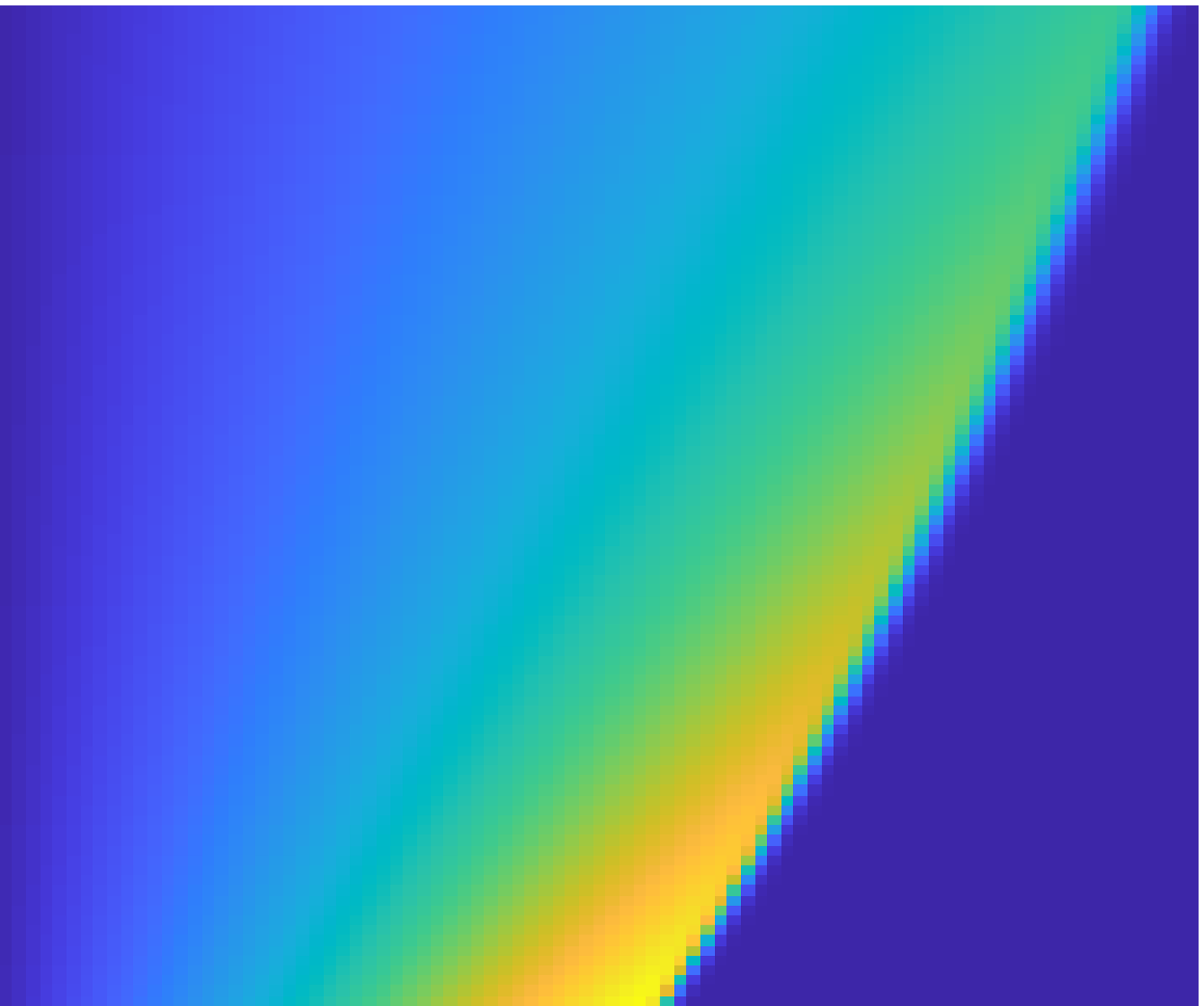};
			\end{groupplot}
		\end{tikzpicture}
		\caption{snapshot data}
		\label{fig:Burgers:FOM}
	\end{subfigure}\hfill
	\begin{subfigure}[t]{.32\linewidth}
		\centering
		\begin{tikzpicture}
			\pgfplotsset{width=5cm,height=5cm}
			\begin{groupplot}[group style={group size=1 by 1, horizontal sep=2.5cm}]
				\nextgroupplot[enlargelimits=false,axis on top=false,xlabel=$x$,ylabel=$t$,ylabel style={rotate=-90},xtick=\empty,ytick=\empty]
				\addplot graphics [xmin=0,xmax=1,ymin=0,ymax=1]{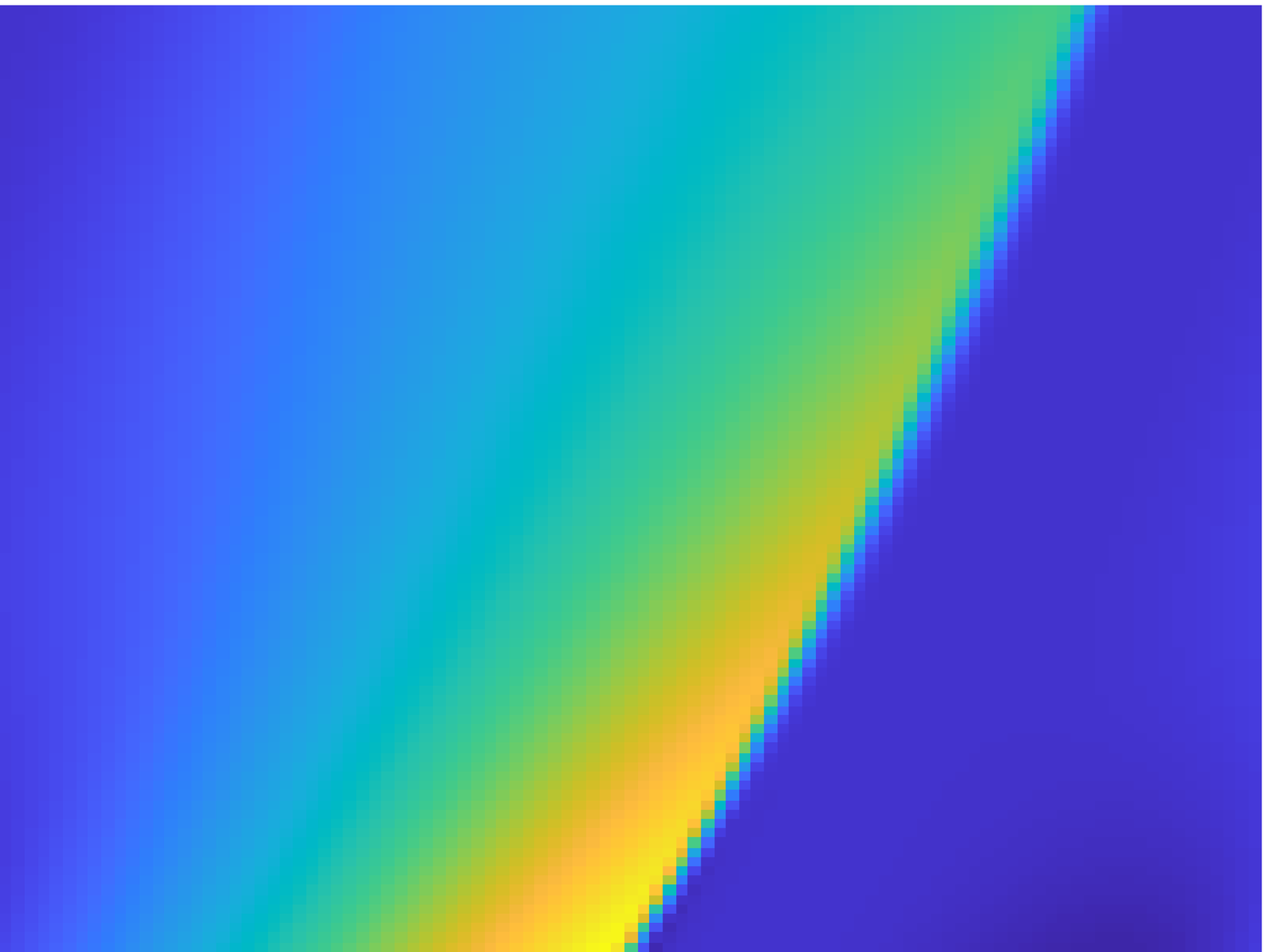};
				\input{burgers_path_optimized}
			\end{groupplot}
		\end{tikzpicture}
		\caption{sPOD $\ROMdim=2$}
		\label{fig:Burgers:sPODapprox}
	\end{subfigure}\hfill
	\begin{subfigure}[t]{.32\linewidth}
		\centering
		\begin{tikzpicture}
			\pgfplotsset{width=5cm,height=5cm}
			\begin{groupplot}[group style={group size=1 by 1, horizontal sep=2.5cm}]
				\nextgroupplot[enlargelimits=false,axis on top=false,xlabel=$x$,ylabel=$t$,ylabel style={rotate=-90},xtick=\empty,ytick=\empty]
				\addplot graphics [xmin=0,xmax=1,ymin=0,ymax=1]{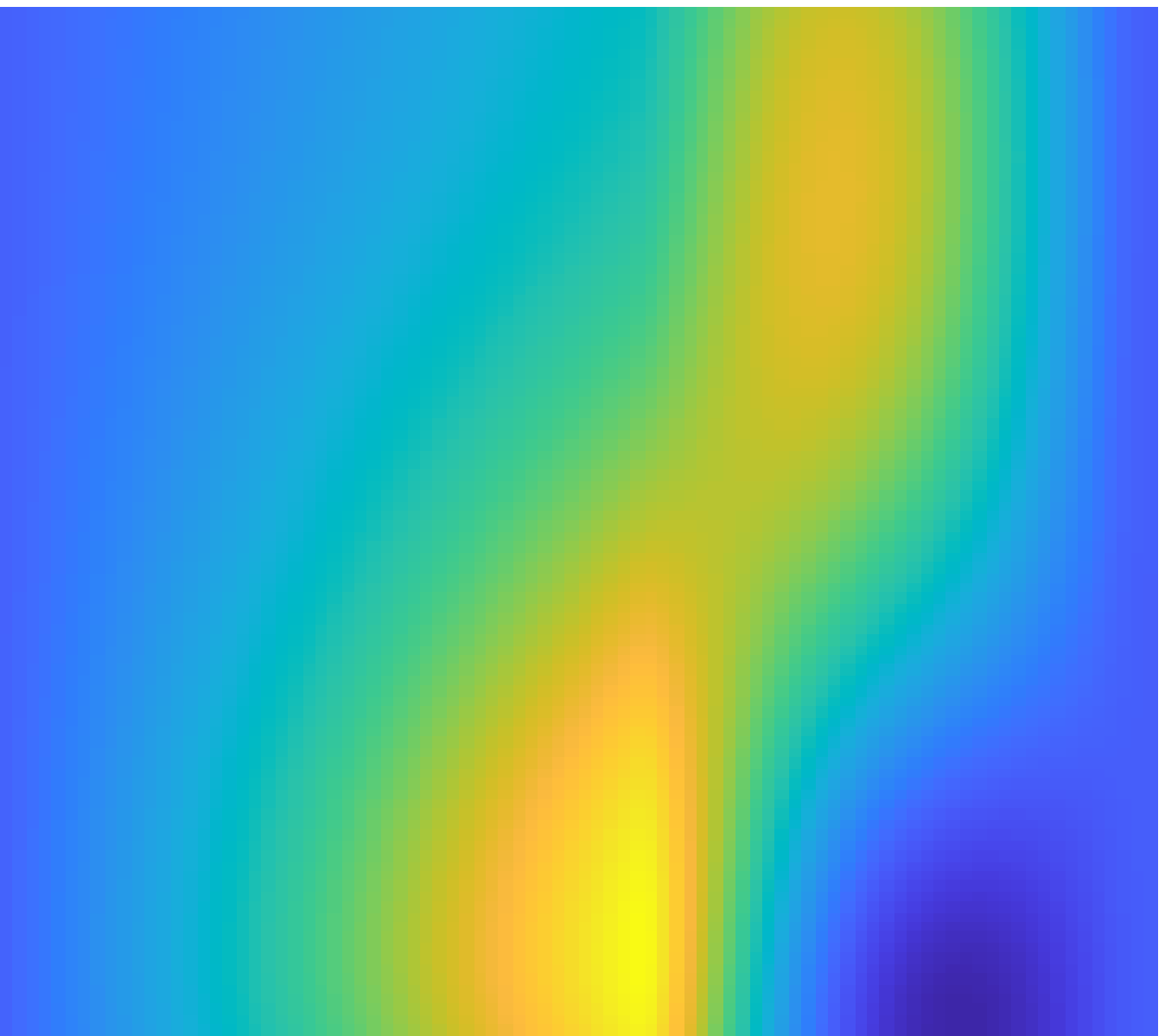};
			\end{groupplot}
		\end{tikzpicture}
		\caption{POD $\ROMdim=2$}
		\label{fig:Burgers:PODapprox}
	\end{subfigure}
	\caption{Burgers' equation -- comparison of original data and approximations }
	\label{fig:result_burgers}
\end{figure}
Besides, we compare the relative $L^2$ errors of the approximations for different mode numbers. For the
initialization of the optimization algorithm, we use the
first $\ROMdim$ snapshots. The coefficients and the path for the optimization are initialized as before. The resulting errors are presented in Table~\ref{tab:comparison_burgers}, detailing the superior approximation capabilities of our method for this test case.
\begin{table}
	\centering
	\caption{Burgers' equation -- comparison of relative $L^2$ errors}
	\label{tab:comparison_burgers}
	\begin{tabular}{l@{\hspace{2em}}r@{\hspace{2em}}r}
		\toprule
		$\ROMdim$ & sPOD & POD \\
		\midrule
		\num{1} & \num{1.217e-1} & \num{4.499e-1} \\
		\num{2} & \num{2.887e-2} & \num{2.853e-1} \\
		\num{3} & \num{1.312e-2} & \num{2.102e-1} \\
		\num{4} & \num{8.406e-3} & \num{1.654e-1} \\
		\num{5} & \num{6.565e-3} & \num{1.347e-1} \\
		\bottomrule
	\end{tabular}
\end{table}

\subsection{Nonlinear Schrödinger equation}
In this section, we consider the nonlinear Schrödinger equation
\begin{align*}
	\mathrm{i} \tfrac{\partial}{\partial t} z(x,t) &= -\tfrac{1}{2} \tfrac{\partial^2}{\partial x^2} z(t,x) + \kappa \vert z(t,x) \vert^{2} z(t,x), \\
	z(0,t) &= 2 \sech (x+7) \exp(2 \mathrm{i} x) + 2 \sech(x-7) \exp(-2 \mathrm{i} x),
\end{align*}
as presented in \cite[Example 4]{MenBALK20}. We compute a solution using the code from \cite{MenBALK20} on a uniform
grid of $501\times 1024$ points on the domain $\mathbb{T} \times \Omega =  [0,2\pi]\times [-15, 15]$. The absolute value
of the numerical solution is presented in Figure~\ref{fig:Schroedinger:FOM}.

We initialize our algorithm by assuming an approximation with two frames, each with two modes. As
starting value for the modes, we use
\begin{align*}
	\mode_{\rho,1}(x) = 2 \sech(x - (-1)^\rho 7) \exp(-(-1)^\rho 2 \mathrm{i} x)\qquad\text{and}\qquad
	\mode_{\rho,2} \equiv 0
\end{align*}
for $\rho=1,2$.
The coefficients are initialized as constants, with value $1$ at each time point. 
For the initial paths, we use $p_1(t) = 2t$ and $p_2(t) = -2t$. 
The corresponding approximation and the absolute error are presented in Figures~\ref{fig:Schroedinger:sPODapprox_linear} and~\ref{fig:Schroedinger:error_linear}.
We observe that the error
results mainly from the complicated wave dynamics in the middle of the spatial and time domain,
whereas the transported wave profiles are captured accurately. Let us emphasize that the error is very localized such that it can be captured with only a few additional POD modes.
\begin{figure}
	\centering
	\begin{subfigure}[t]{.32\linewidth}
		\centering
		\begin{tikzpicture}
			\pgfplotsset{width=5cm,height=5cm}
			\begin{groupplot}[group style={group size=1 by 1, horizontal sep=2.5cm}]
				\nextgroupplot[enlargelimits=false,axis on top=false,xlabel=$x$,ylabel=$t$,ylabel style={rotate=-90},xtick=\empty,ytick=\empty]
				\addplot graphics [xmin=0,xmax=1,ymin=0,ymax=1]{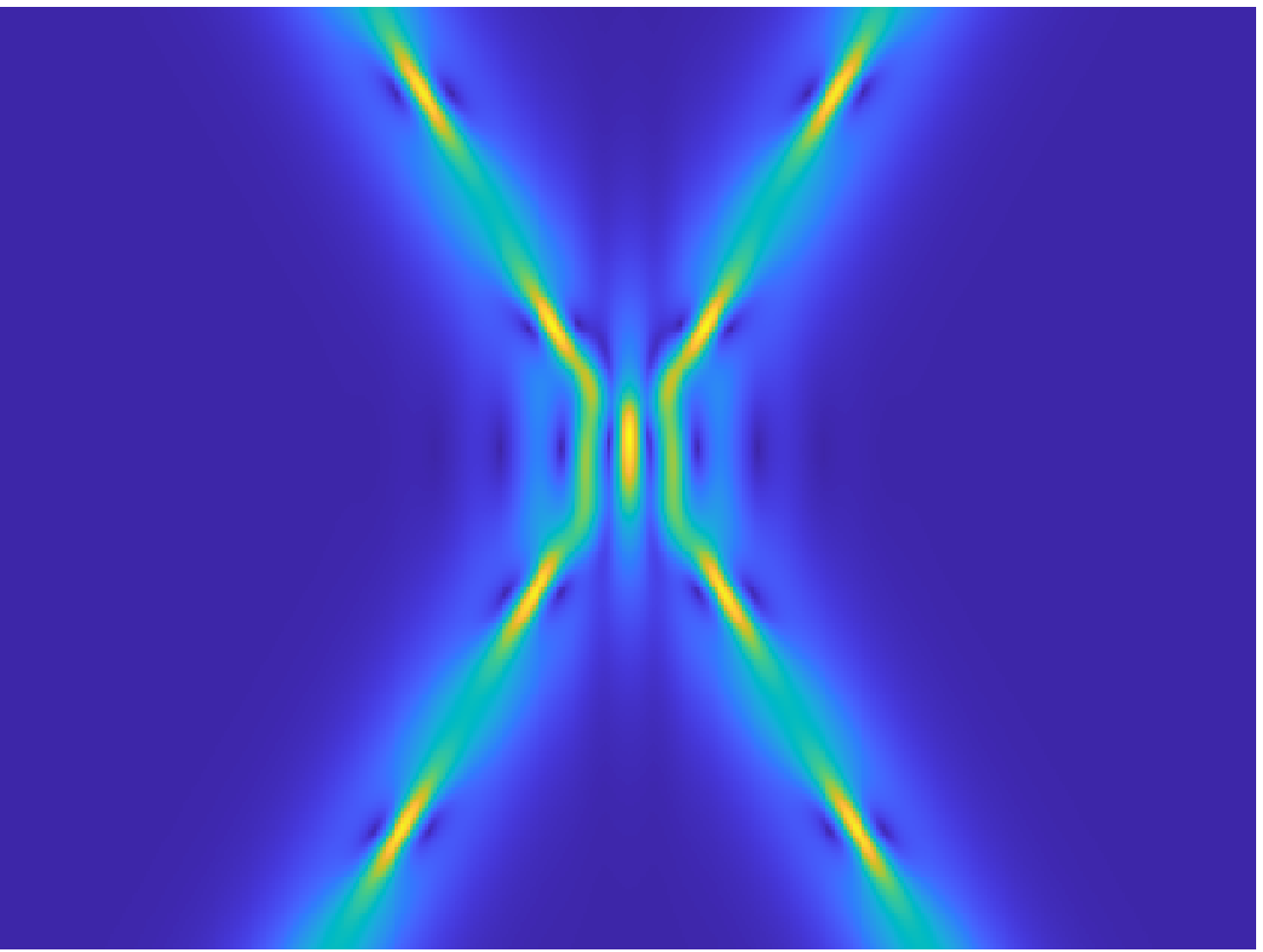};
			\end{groupplot}
		\end{tikzpicture}
		\caption{snapshot data}
		\label{fig:Schroedinger:FOM}
	\end{subfigure}\hfill
	\begin{subfigure}[t]{.32\linewidth}
		\centering
		\begin{tikzpicture}
			\pgfplotsset{width=5cm,height=5cm}
			\begin{groupplot}[group style={group size=1 by 1, horizontal sep=2.5cm}]
				\nextgroupplot[enlargelimits=false,axis on top=false,xlabel=$x$,ylabel=$t$,ylabel style={rotate=-90},xtick=\empty,ytick=\empty]
				\addplot graphics
				[xmin=0,xmax=1,ymin=0,ymax=1]{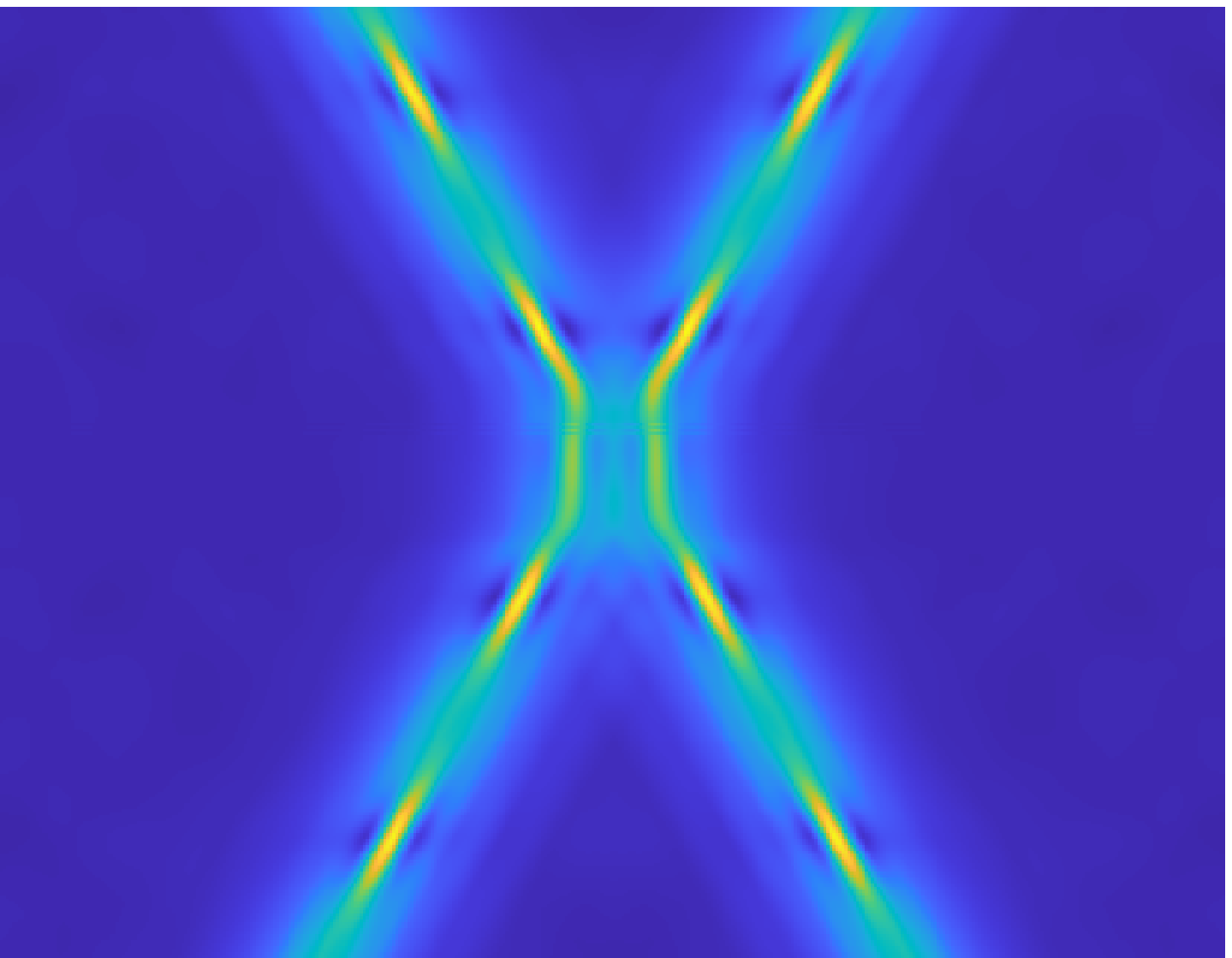};
				\input{schroedinger_linear_path_optimized}
			\end{groupplot}
		\end{tikzpicture}
		\caption{sPOD $\ROMdim=4$}
		\label{fig:Schroedinger:sPODapprox_linear}
	\end{subfigure}\hfill
	\begin{subfigure}[t]{.32\linewidth}
		\centering
		\begin{tikzpicture}
			\pgfplotsset{width=5cm,height=5cm}
			\begin{groupplot}[group style={group size=1 by 1, horizontal sep=2.5cm}]
				\nextgroupplot[enlargelimits=false,axis on top=false,xlabel=$x$,ylabel=$t$,ylabel style={rotate=-90},xtick=\empty,ytick=\empty]
				\addplot graphics
				[xmin=0,xmax=1,ymin=0,ymax=1]{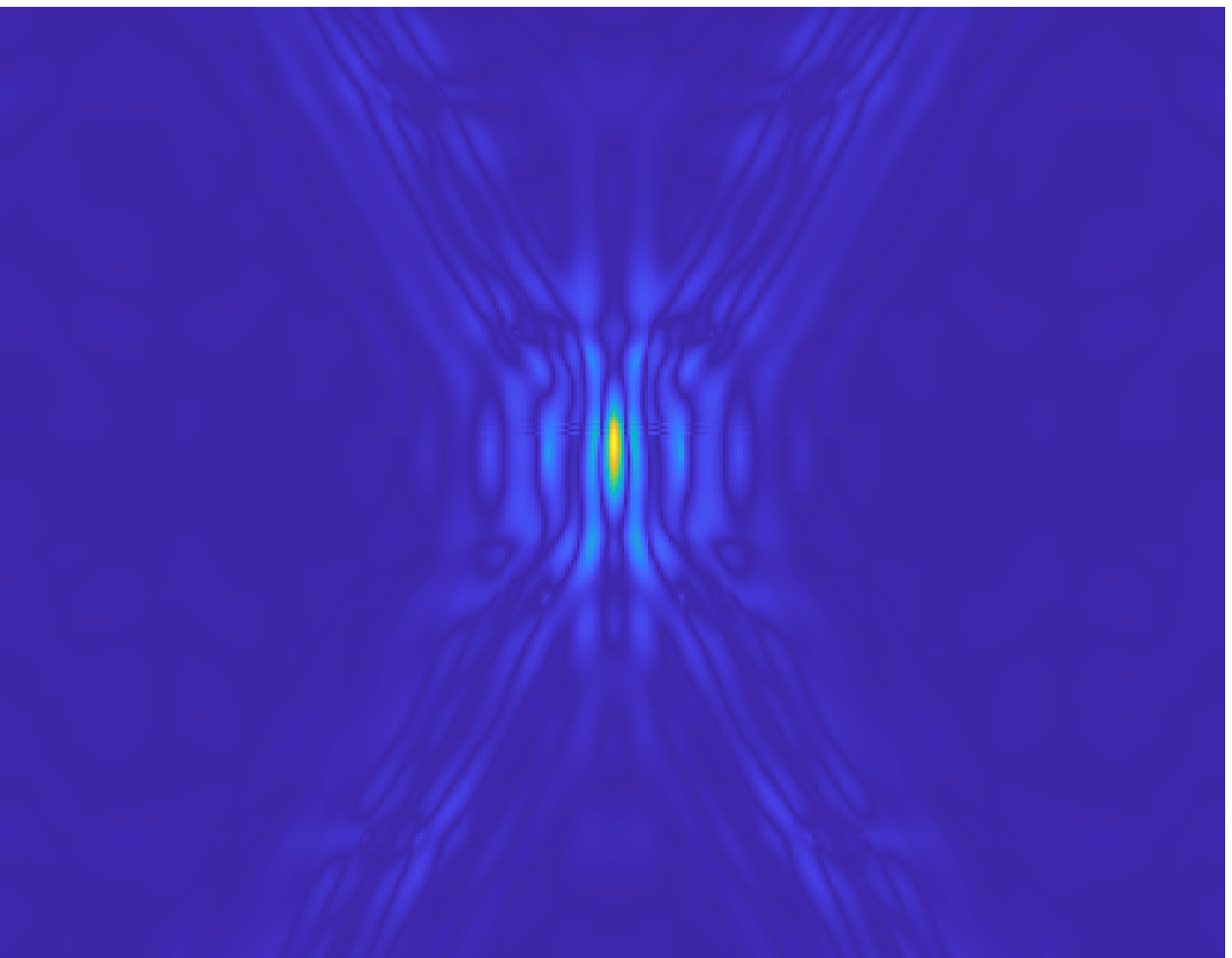};
			\end{groupplot}
		\end{tikzpicture}
		\caption{absolute error}
		\label{fig:Schroedinger:error_linear}
	\end{subfigure}
	\caption{Nonlinear Schrödinger equation -- original data, approximation, and absolute error for initial paths $p_1(t) = 2t$ and $p_2(t) = -2t$}
	\label{fig:Schroedinger}
\end{figure}
We notice that the optimizer does not keep the linear path over the whole time domain.  Instead, as depicted in Figure~\ref{fig:Schroedinger:sPODapprox_linear}, in the
middle of the computational domain, the paths jump between the wavefronts. Inspecting the snapshot matrix in the co-moving frame along the path $p_1(t) = 2t$ in Figure~\ref{fig:schroedinger_ref_frame_1} provides a possible explanation: the vertical wavefront features an offset after the two waves have crossed. The optimizer needs to
account for this offset, which explains the jump. 
\begin{figure}
	\centering
	\begin{tikzpicture}
		\pgfplotsset{width=12cm,height=5cm}
			\begin{groupplot}[group style={group size=1 by 1, horizontal sep=2.5cm}]
				\nextgroupplot[enlargelimits=false,axis on top=false,xlabel=$x$,ylabel=$t$,ylabel style={rotate=-90},xtick=\empty,ytick=\empty]
				\addplot graphics
				[xmin=0,xmax=1,ymin=0,ymax=1]{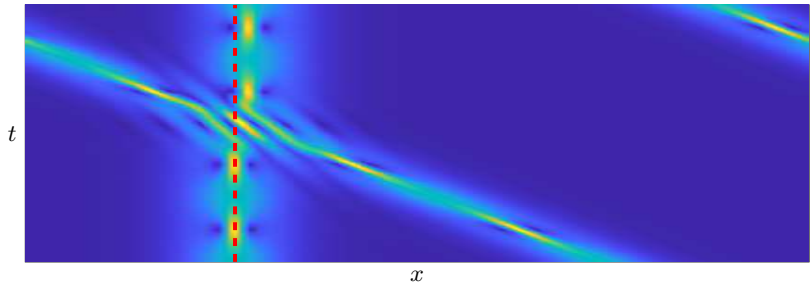};
				\input{schroedinger_linear_path_ref_frame_1_ref_line}
			\end{groupplot}
		\end{tikzpicture}
		\caption{Nonlinear Schrödinger equation -- transformation to the co-moving frame along $p_1(t) = 2t$}
		\label{fig:schroedinger_ref_frame_1}
	\label{fig:schroedinger_ref_frame}
\end{figure}
Let us emphasize that with a different initialization for the path variables, the optimizer finds another local minimum with a similar approximation quality. Using piecewise linear paths as depicted in Figure~\ref{fig:schroedinger_piecewise_initial}, we observe that the
resulting optimized path smoothes out the edges of the initial path in the middle of the domain (cf.~Figure~\ref{fig:schroedinger_piecewise_initial_zoom}) and does not feature any jumps. 
It is smooth and tracks the wavefronts as if the waves reflect off each other.
\begin{figure}
	\begin{subfigure}[t]{.32\linewidth}
		\centering
		\begin{tikzpicture}
			\pgfplotsset{width=5cm,height=5cm}
			\begin{groupplot}[group style={group size=1 by 1, horizontal sep=2.5cm}]
				\nextgroupplot[enlargelimits=false,axis on top=false,xlabel=$x$,ylabel=$t$,ylabel style={rotate=-90},xtick=\empty,ytick=\empty]
				\addplot graphics
				[xmin=0,xmax=1,ymin=0,ymax=1]{schroedinger_snapshot_new};
				\input{schroedinger_constructed_path_initial}
			\end{groupplot}
		\end{tikzpicture}
		\caption{initial paths}
		\label{fig:schroedinger_piecewise_initial}
	\end{subfigure} \hfill
	\begin{subfigure}[t]{.32\linewidth}
		\centering
		\begin{tikzpicture}
			\pgfplotsset{width=5cm,height=5cm}
			\begin{groupplot}[group style={group size=1 by 1, horizontal sep=2.5cm}]
				\nextgroupplot[enlargelimits=false,axis on top=false,xlabel=$x$,ylabel=$t$,ylabel style={rotate=-90},xtick=\empty,ytick=\empty]
				\addplot graphics
				[xmin=0,xmax=1,ymin=0,ymax=1]{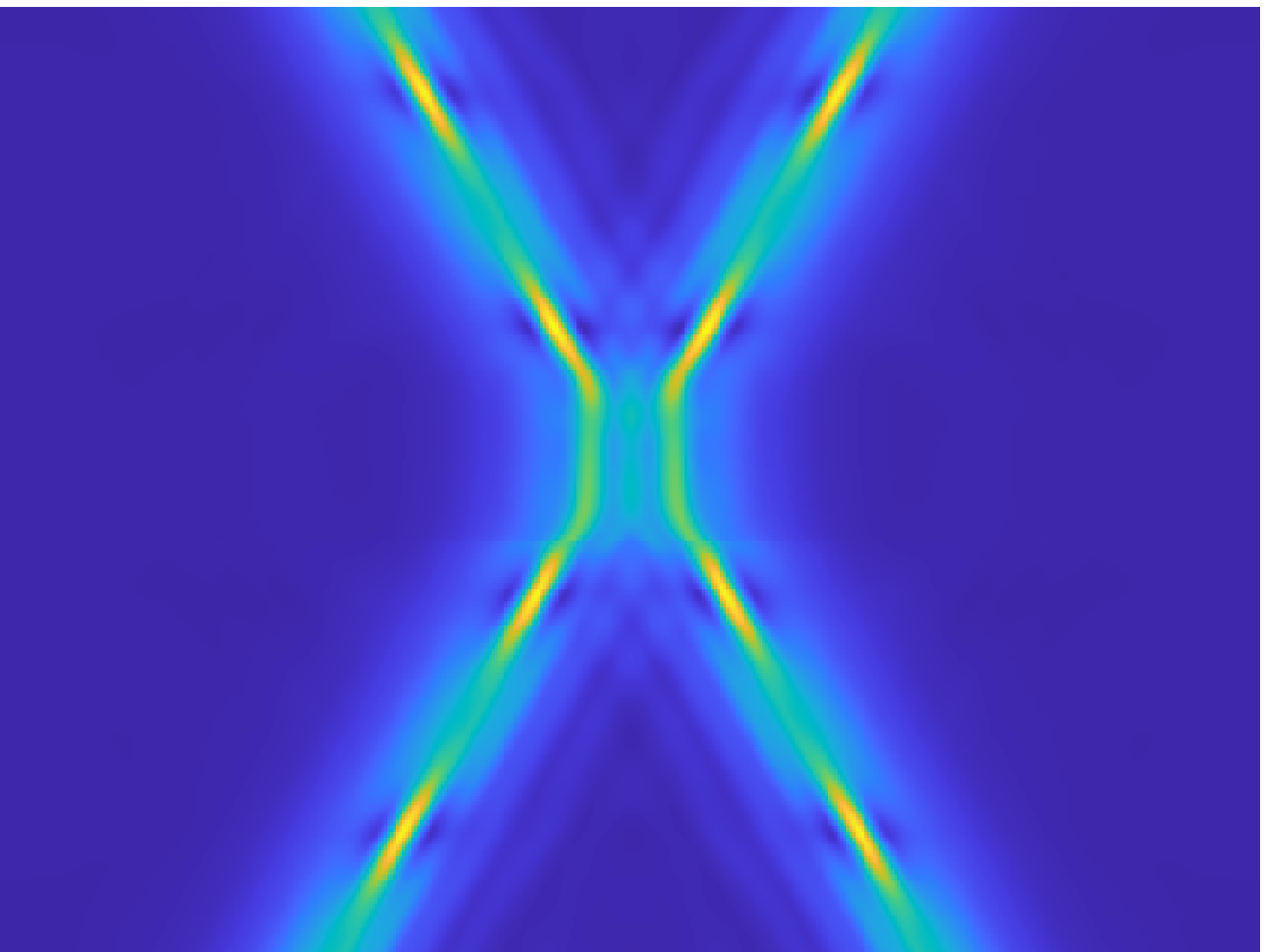};
				\input{schroedinger_constructed_path_optimized}
			\end{groupplot}
		\end{tikzpicture}
		\caption{sPOD $\ROMdim=4$}
		\label{fig:Schroedinger:sPODapprox_constructed}
	\end{subfigure} \hfill
	\begin{subfigure}[t]{.32\linewidth}
		\centering
		\begin{tikzpicture}
			\pgfplotsset{width=5cm,height=5cm}
			\begin{groupplot}[group style={group size=1 by 1, horizontal sep=2.5cm}]
				\nextgroupplot[enlargelimits=false,axis on top=false,xlabel=$x$,ylabel=$t$,ylabel style={rotate=-90},xtick=\empty,ytick=\empty]
				\addplot graphics
				[xmin=0.4,xmax=0.6,ymin=0.3,ymax=0.71]{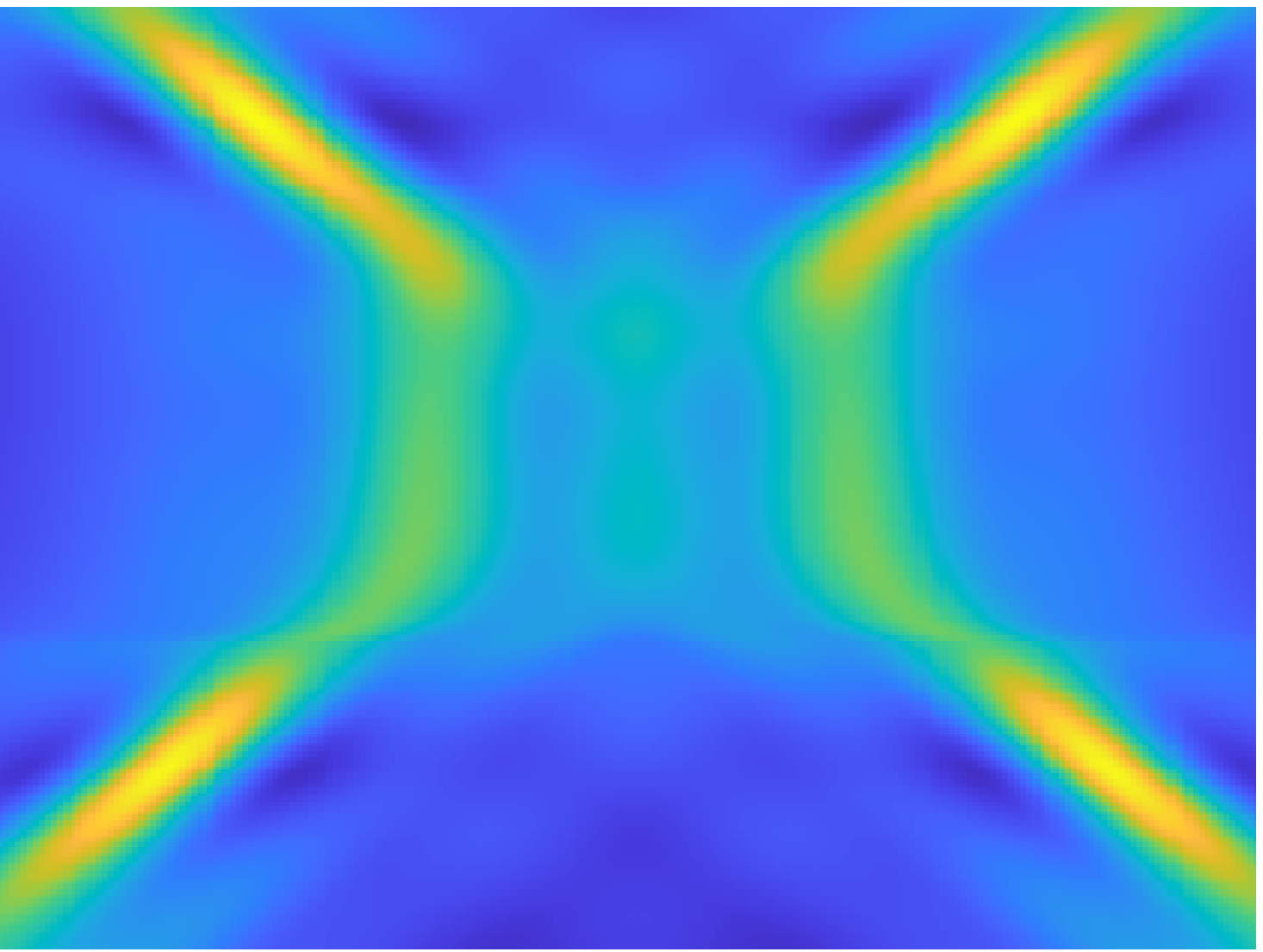};
				\input{schroedinger_constructed_path_initial_zoom_new}
				\input{schroedinger_constructed_path_optimized_zoom_new}
			\end{groupplot}
		\end{tikzpicture}
		\caption{zoom: initial vs.~optimized path}
		\label{fig:schroedinger_piecewise_initial_zoom}
	\end{subfigure}
	\caption{Nonlinear Schrödinger equation -- original data and approximation for a piecewise linear initialization of the path variables as depicted in Figure~\ref{fig:schroedinger_piecewise_initial}}
	\label{fig:Schroedinger:piecewiseInitial}
\end{figure}

\subsection{FitzHugh--Nagumo wave train}

We follow \cite{Koc21} and consider the FitzHugh--Nagumo model given by
\begin{equation}
	\label{eqn:FitzHughNagumo}
	\begin{aligned}
		\tfrac{\partial}{\partial t} u(t,x) &= \nu \tfrac{\partial^2}{\partial x^2}u(t,x)-v(t,x)+u(t,x)(1-u(t,x))(u(t,x)-a),\\
		\tfrac{\partial}{\partial t} v(t,x) &=\epsilon(bu(t,x)-v(t,x)),
	\end{aligned}
\end{equation}
with spatial domain $(0,500)$ and time interval $(0,1000)$. 
The partial differential equation \eqref{eqn:FitzHughNagumo} is closed by periodic boundary conditions and the initial condition
\begin{equation*}
	u(0,x) = \tfrac12\left(1+\sin\left(\tfrac{\pi}{50}x\right)\right),\quad v(0,x) = \tfrac12\left(1+\cos\left(\tfrac{\pi}{50}x\right)\right).
\end{equation*}
For the parameter values, we choose $\nu=1$, $a=-0.1$, $\epsilon=0.05$, and $b=0.3$.
The spatial discretization of \eqref{eqn:FitzHughNagumo} is performed via a central sixth-order finite-difference scheme with mesh width $h=0.5$ and for the time integration we use \matlab's \texttt{ode45} function based on a time grid with step size $1$.
The corresponding numerical solution for the variable $u$ is depicted in Figure~\ref{fig:FitzHughNagumo:FOM}.

For the optimization we consider only the data of the variable $u$ and use an approximation with one reference frame to account for the traveling wave train.
Furthermore, we reduce the computational complexity by considering the optimization problem only in terms of the path, whereas the coefficients and modes are computed in each iteration via a truncated singular value decomposition of the snapshot matrix shifted into the co-moving reference frame.
Here we exploit that the periodic shift operator is isometric, such that we can solve the optimization problem via classical POD with transformed data, cf.~\cite[Thm.~4.8]{BlaSU20}.
As starting value for the path, we choose a linear function in $t$ with slope $1.04$, which we determined by inspecting the first and the last snapshot of the original data.
The corresponding approximation obtained from the optimization procedure is depicted in Figure~\ref{fig:FitzHughNagumo:sPODapprox}.
As reference approximation, we consider a POD approximation with the same number of modes in Figure~\ref{fig:FitzHughNagumo:PODapprox}.
The corresponding total relative errors are $15\%$ for the approximation based on shifted modes and $31\%$ for the POD approximation.
We note that in contrast to the Burgers test case considered in section~\ref{sec:Burgers}, the traveling wave train can be better approximated by POD due to the lack of a traveling shock wave.
Correspondingly, the difference between the POD approximation and the approximation based on shifted modes for the considered FitzHugh--Nagumo test case is less striking than the one observed for the example in section~\ref{sec:Burgers}.

\begin{figure}
	\centering
	\begin{subfigure}[t]{.32\linewidth}
		\centering
		\begin{tikzpicture}
			\pgfplotsset{width=5cm,height=5cm}
			\begin{groupplot}[group style={group size=1 by 1, horizontal sep=2.5cm}]
				\nextgroupplot[enlargelimits=false,axis on top=false,xlabel=$x$,ylabel=$t$,ylabel style={rotate=-90},xtick=\empty,ytick=\empty]
				\addplot graphics [xmin=0,xmax=500,ymin=0,ymax=1000]{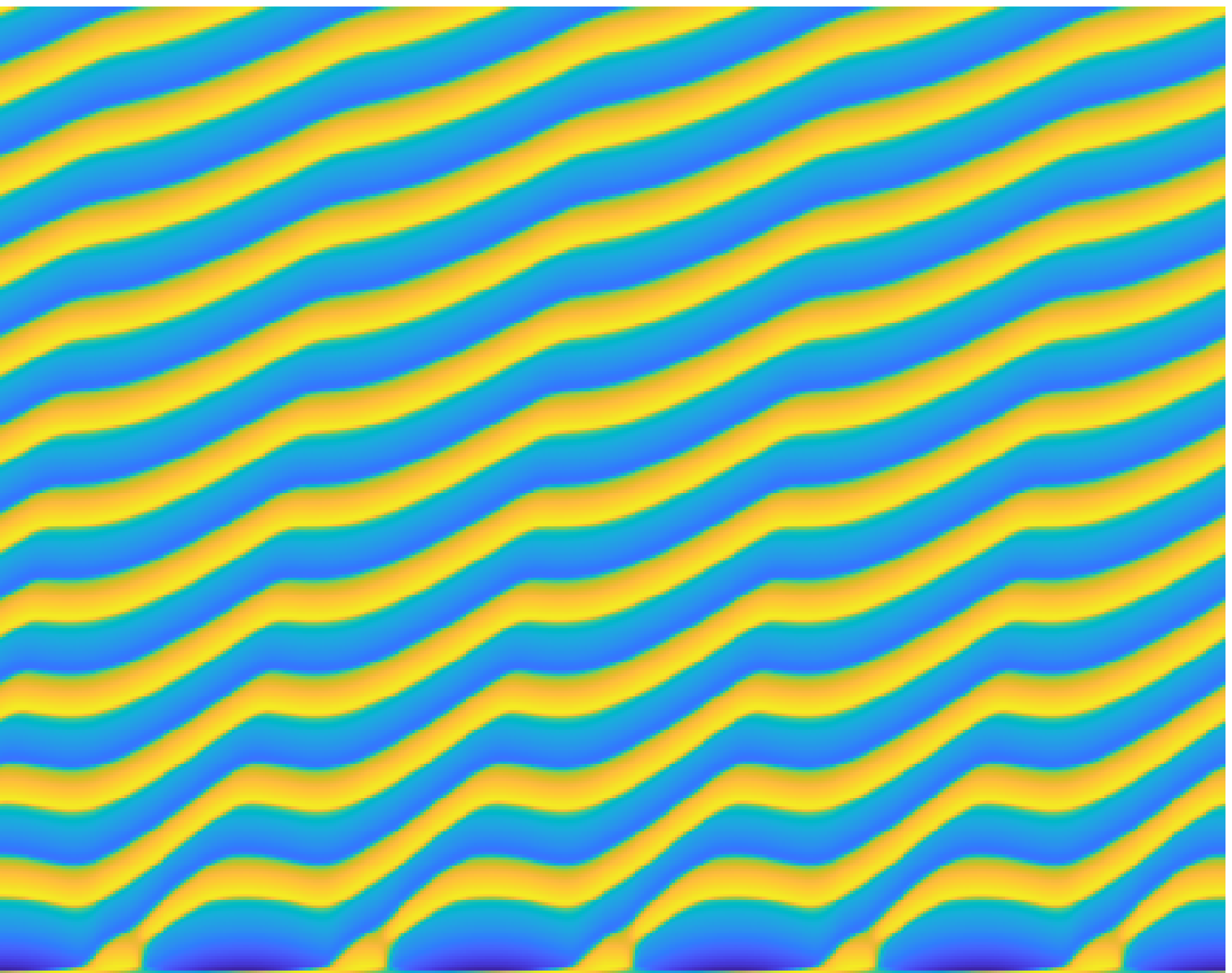};
			\end{groupplot}
		\end{tikzpicture}
		\caption{snapshot data}
		\label{fig:FitzHughNagumo:FOM}
	\end{subfigure}\hfill
	\begin{subfigure}[t]{.32\linewidth}
		\centering
		\begin{tikzpicture}
			\pgfplotsset{width=5cm,height=5cm}
			\begin{groupplot}[group style={group size=1 by 1, horizontal sep=2.5cm}]
				\nextgroupplot[enlargelimits=false,axis on top=false,xlabel=$x$,ylabel=$t$,ylabel style={rotate=-90},xtick=\empty,ytick=\empty]
				\addplot graphics [xmin=0,xmax=500,ymin=0,ymax=1000]{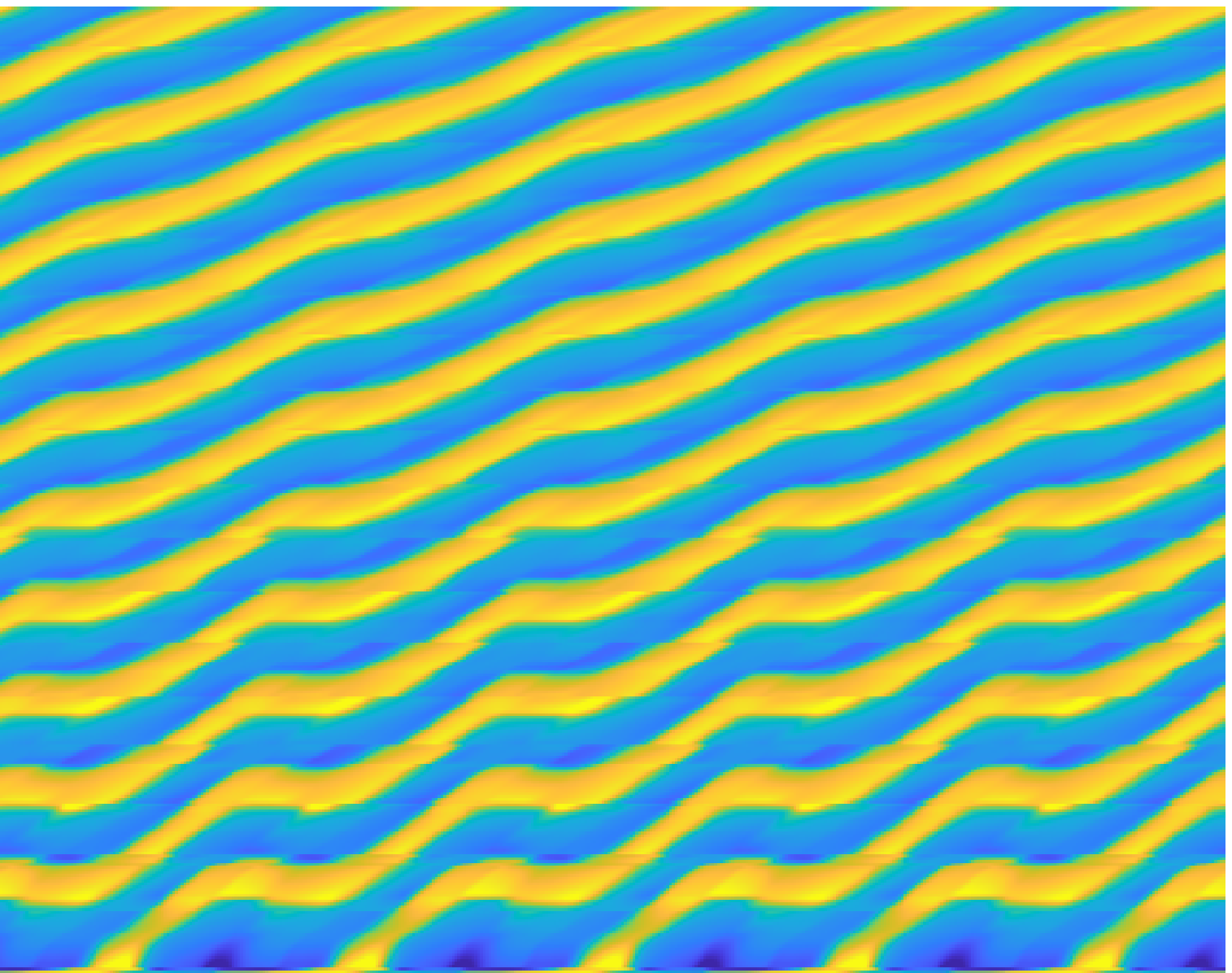};
				\input{FitzHughNagumo_path}
			\end{groupplot}
		\end{tikzpicture}
		\caption{sPOD $\ROMdim=4$}
		\label{fig:FitzHughNagumo:sPODapprox}
	\end{subfigure}\hfill
	\begin{subfigure}[t]{.32\linewidth}
		\centering
		\begin{tikzpicture}
			\pgfplotsset{width=5cm,height=5cm}
			\begin{groupplot}[group style={group size=1 by 1, horizontal sep=2.5cm}]
				\nextgroupplot[enlargelimits=false,axis on top=false,xlabel=$x$,ylabel=$t$,ylabel style={rotate=-90},xtick=\empty,ytick=\empty]
				\addplot graphics [xmin=0,xmax=500,ymin=0,ymax=1000]{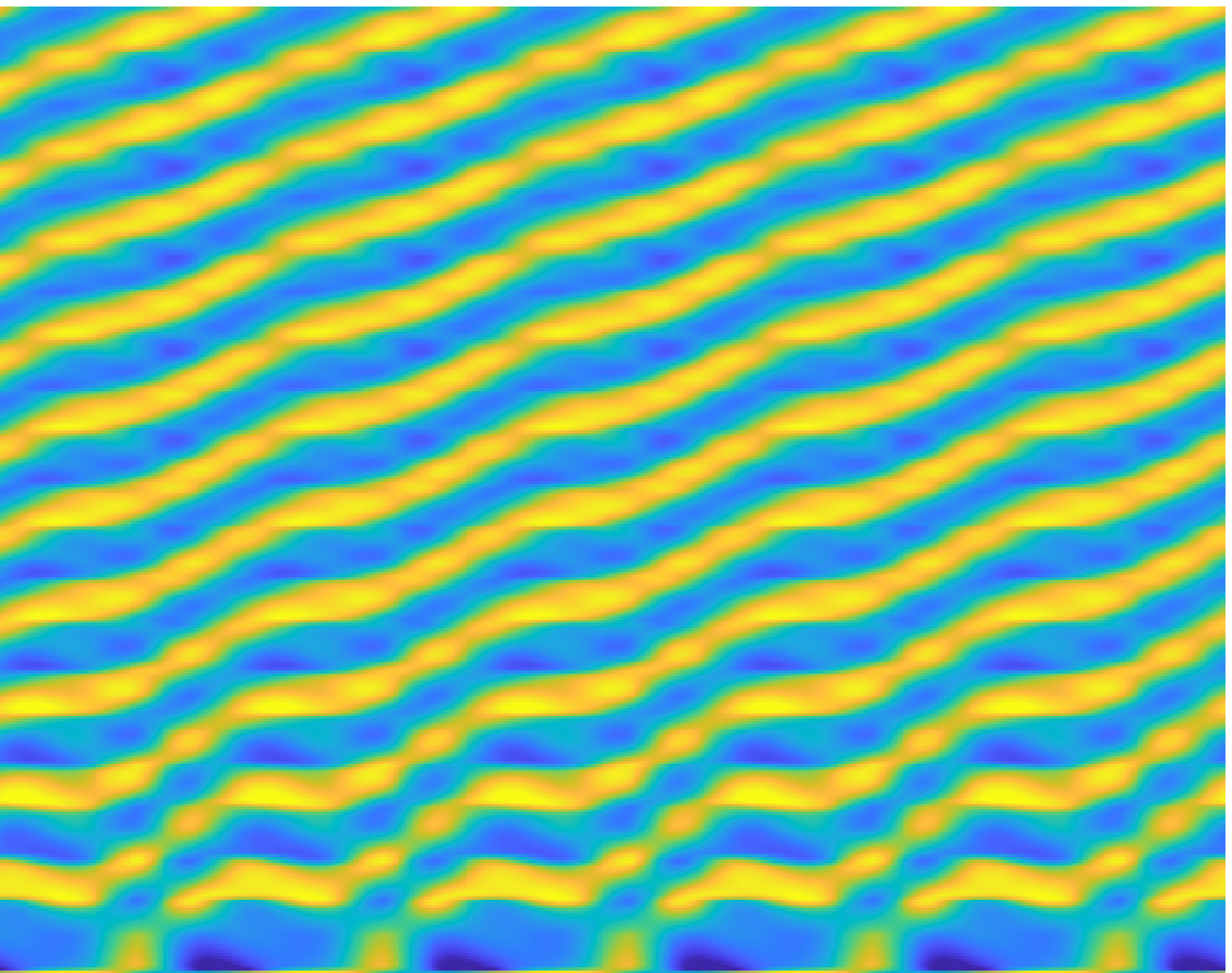};
			\end{groupplot}
		\end{tikzpicture}
		\caption{POD $\ROMdim=4$}
		\label{fig:FitzHughNagumo:PODapprox}
	\end{subfigure}
	\caption{FitzHugh--Nagumo model -- comparison of original data and approximations}
	\label{fig:FitzHughNagumo}
\end{figure}

\section{Summary}
In this paper, we analyze the problem of determining an optimal approximation of given snapshot data by a linear combination of dynamically transformed modes.
This data compression can, for instance, be used for model order reduction of transport-dominated systems, see, for example, \cite{BlaSU20} for an approach that makes use of such decompositions to construct dynamical ROMs via projection.
As optimization parameters, we consider the modes, the corresponding coefficients or amplitudes, and the so-called path variables, which parameterize the coordinate transforms applied to the modes.
We first show that the considered infinite-dimensional optimization problem possesses a minimizing solution if the admissible set is constrained such that the optimization parameters are norm bounded.
Afterward, we derive a corresponding unconstrained optimization problem by adding an appropriate penalization term and
show that the unconstrained problem also has a solution.
Furthermore, we demonstrate that if the penalization coefficient tends to infinity, each limit point of the corresponding sequence of minimizers is a solution to the original constrained optimization problem.
To derive a gradient-based optimization procedure, we compute the partial Fréchet derivatives of the unconstrained cost functional and discuss their space and time discretization.
Finally, we apply the optimization procedure to some numerical test cases and observe that the optimized decompositions
are significantly more accurate than corresponding approximations obtained by the classical proper orthogonal
decomposition with the same number of modes. 

After full discretization, the optimization problem still features a large number of optimization parameters scaling with the number of grid points in space and time.
Thus, an interesting future research direction is to investigate approaches for reducing the computational complexity of
the optimization procedure, for instance, by using multigrid optimization techniques \cite{Nas00}, or by making use of low-dimensional parametrizations of the optimization parameters, cf.~Remark~\ref{rem:pathParameterization}.
Such a parametrization seems to be especially promising for reducing the complexity in the path variables since our numerical experiments revealed that the optimization procedure is sensitive with respect to the paths.
Furthermore, let us emphasize that, although we have only discussed applications in a one-dimensional spatial domain with a periodic shift operator, our framework is not restricted to this case. Thus, another promising direction for the future is to explore the applicability to problems with higher-dimensional spatial domains using different transformation operators, see e.g.~\cite{KraSR21,RimPM20,Tad20} for some contributions in this direction.
 
\paragraph{Acknowledgments}
The work of F.~Black is supported by the Deutsche Forschungsgemeinschaft (DFG, German Research Foundation) Collaborative Research Center (CRC) 1029 \emph{Substantial efficiency increase in gas turbines through direct use of coupled unsteady combustion and flow dynamics}, project number 200291049. 
P.~Schulze acknowledges funding by the DFG CRC Transregio 154 \emph{Mathematical Modelling, Simulation and Optimization Using the Example of Gas Networks}, project number 239904186. 
B.~Unger acknowledges funding from the DFG under Germany's Excellence Strategy -- EXC 2075 -- 390740016 and is thankful for support by the Stuttgart Center for Simulation Science (SimTech).

\bibliographystyle{plain}
\bibliography{literature}

\appendix

\section{Shifted Inner Products of Hat Functions}
\label{sec:shiftedInnerProducts}

In this section, we explicitly compute the path-dependent inner products for a particular example.  Let us assume a one-dimensional domain $\Omega=(0,1)$, which we discretize with an equidistant grid with step size $h\vcentcolon=\tfrac{1}{n}$ for some $n\in\N$.  We consider the spaces $\BS = L^2(\Omega)$ and $\BSS = H^1_{\mathrm{per}}(\Omega)$.  
For the discretization of $\BSS$ we choose periodic $P_1$ finite elements, which are given as
\begin{align*}
	&\psi_\ell(x) = 
	\begin{cases}
		\frac{x-(\ell-1)h}h 	& \text{for } x\in ((\ell-1)h,\ell h],\\
		\frac{(\ell+1)h-x}h 	& \text{for } x\in (\ell h,(\ell+1)h),\\
		0 						& \text{otherwise},
	\end{cases}
	\quad \text{for } \ell = 1,\ldots,n-1,\\
	\text{and}\quad
	&\psi_n(x) = 
	\begin{cases}
		\frac{x-(n-1)h}h 	& \text{for } x\in ((n-1)h,nh],\\
		\frac{h-x}h 		& \text{for } x\in (0,h),\\
		0 						& \text{otherwise},
	\end{cases}
\end{align*}
with derivatives
\begin{align*}
	&\psi^\prime_k = 
	\begin{cases}
		\frac1h 		& \text{for } x\in ((k-1)h,kh),\\
		-\frac1h 	& \text{for } x\in (kh,(k+1)h),\\
		0 				& \text{otherwise},
	\end{cases}
	\quad \text{for } k = 1,\ldots,n-1,\\
	\text{and }
	&\psi^\prime_n = 
	\begin{cases}
		\frac1h 		& \text{for } x\in ((n-1)h,nh),\\
		-\frac1h		& \text{for } x\in (0,h),\\
		0 				& \text{otherwise}.
	\end{cases}
\end{align*}
For the family of transformation operators we choose the shift operator with periodic embedding, cf.~Example~\ref{ex:counterExampleDirectionalDerivative}, i.e., $\transform_i(\path_i)\mode_i = \mode_i(\cdot-\path_i)$.

\begin{lemma}
	Consider the settings as described above. Then for $i,j=1,\ldots,\ROMdim$, $\mode_i\in\BSS$, and $\path_i\in\R$ we have
	\begin{gather*}
		 \tfrac{\partial}{\partial \path_i} \transform_i(\path_i)\mode_i = -\transform_i(\path_i)\tfrac{\partial}{\partial x} \mode_i, \qquad M_{i,j}(\path_i,\path_j) = F_i(\path_i-\path_j),\\
		 N_{i,j}(\path_i,\path_j) = G_i(\path_i-\path_j),
	\end{gather*}
	with $M_{i,j}$, $N_{i,j}$, $F_i$, and $G_i$ defined in~\eqref{eqn:pathDependentInnerProducts}.
\end{lemma}

\begin{proof}
	The statement is a mere consequence of the fact that the shift operator is unitary and a semi-group and $-\tfrac{\partial}{\partial x}$ is its generator, see for instance~\cite[Sec.~II.2.10]{EngN00}.  \qed
\end{proof}

It is thus sufficient to compute the inner products for $F_i$ and $G_i$. We first observe that for $\path_i = qh$ for some $q\in\Z$ we obtain
\begin{align*}
	\transform_i(\path_i)\psi_\ell = \psi_\ell(\cdot-qh) = \psi_{\ell+q} = \transform_i(0)\psi_{\ell+q}
\end{align*}
with the understanding that for $\ell+q\not\in\{1,\ldots,n\}$ we use $((\ell+q-1)\mod n)+1$ instead. 
It is thus sufficient to compute the inner products for $\path_i\in[0,h)$.

\begin{lemma}
	Consider the setting as described above. Then for $\path_i\in[0,h)$ we obtain
	\begin{align*}
		\resizebox{.95\linewidth}{!}{$
		\langle \psi_k,\transform_i(\path_i)\psi_\ell\rangle = \begin{cases}
			\tfrac{1}{h^2}\left(\tfrac{2}{3}(h-\path_i)^3+\path_i(h-\path_i)^2+\path_ih(h-\path_i)+\tfrac16\path_i^3\right), & \text{if } \ell = k,\\
			 \tfrac{1}{6h^2}(h-\path_i)^3, & \text{if } \ell = k+1,\\
			 \tfrac{1}{h^2}\left(\tfrac{1}{6}(h-\path_i)^3 - \tfrac{1}{3}\path_i^3 + h^2\path_i\right),  & \text{if } \ell = k-1,\\			 
			 \tfrac{1}{6h^2} \path_i^3, & \text{if } \ell = k-2,\\
			 0, & \text{otherwise},
		\end{cases}$}
	\end{align*}
	and
	\begin{align*}
		\langle \psi_k,\tfrac{\partial}{\partial \path_i} \transform_i(\path_i)\psi_\ell\rangle = \begin{cases}
			-\tfrac{1}{h^2}\left(2\path_i h-\tfrac{3}{2}\path_i^2\right), & \text{if } \ell=k,\\
			-\tfrac{1}{2h^2}(h-\path_i)^2, & \text{if } \ell=k+1,\\
			-\tfrac{1}{h^2}\left(2\path_i^2-2h\path_i-\tfrac{1}{2}(h-\path_i)^2\right), & \text{if } \ell=k-1,\\
			\tfrac{1}{2h^2}\path_i^2, & \text{if } \ell=k-2,\\
			0, & \text{otherwise.}
		\end{cases}
	\end{align*}
\end{lemma}

\begin{proof}
	We immediately obtain $\langle \psi_k,\transform_i(\path_i)\psi_\ell\rangle_{\BS} = 0$ and $\langle \psi_k,\tfrac{\partial}{\partial \path_i}\transform_i(\path_i)\psi_\ell\rangle_{\BS} = 0$ for $k=1,\ldots,n$ and $\ell\not\in\{k-2,k-1,k,k+1\}$. 
	For $k=1,\ldots,n$ we compute
	\begin{align*}
		\langle \psi_k&,\transform_i(\path_i)\psi_k\rangle_{\BS} = \int_0^1 \psi_k(x)\psi_k(x-\path_i) \dx\\
		&= \int_{(k-1)h+\path_i}^{k h} \psi_k(x)\psi_k(x-\path_i) \dx + \int_{k h}^{k h+\path_i} \psi_k(x)\psi_k(x-\path_i) \dx\\
		&\phantom{=}\quad  + \int_{k h+\path_i}^{(k+1)h} \psi_k(x)\psi_k(x-\path_i) \dx\\
		&=\int_0^{h-\path_i} \tfrac{x(x+\path_i)}{h^2}\dx + \int_0^{\path_i} \tfrac{(x+h-\path_i)(h-x)}{h^2}\dx - \int_{\path_i-h}^0 \tfrac{(\path_i-x)x}{h^2} \dx\\
		&= \tfrac{1}{h^2}\left(\tfrac{2}{3}(h-p_i)^3+p_i(h-p_i)^2+p_ih(h-p_i)+\tfrac16p_i^3\right).
	\end{align*}
	Similarly, we obtain
	\begin{align*}
		\langle \psi_k,\transform_i(\path_i)\psi_{k+1}\rangle_{\BS} &= \int_{k h+\path_i}^{(k+1)h} \psi_k(x)\psi_{k+1}(x-\path_i)\dx\\
		&= \tfrac{1}{h^2}\int_0^{h-\path_i} (h-\path_i-x)x \dx = \tfrac{1}{6h^2}(h-\path_i)^3.
	\end{align*}
	Furthermore, we have
	\begin{align*}
		\langle \psi_k&,\transform_i(\path_i)\psi_{k-1}\rangle_{\BS}\\
		&= \int_{(k-1)h}^{(k-1)h+\path_i} \psi_{k}(x)\psi_{k-1}(x-\path_i)\dx + \int_{(k-1)h+\path_i}^{k h} \psi_{k}(x)\psi_{k-1}(x-\path_i)\dx\\
		&\phantom{=}\qquad + \int_{k h}^{k h + \path_i} \psi_{k}(x)\psi_{k-1}(x-\path_i)\dx\\
		&= \int_0^{\path_i} \tfrac{(x+h-\path_i)x}{h^2}\dx + \int_0^{h-\path_i} \tfrac{(h-x)(x+\path_i)}{h^2}\dx - \int_{-\path_i}^0 \tfrac{x(h-\path_i-x)}{h^2}\dx\\
		&= \tfrac{1}{h^2}\left(\tfrac{1}{6}(h-\path_i)^3-\tfrac{1}{3}\path_i^3+h^2\path_i\right).
	\end{align*}
	In addition, we obtain
	\begin{align*}
		\langle \psi_k,\transform_i(\path_i)\psi_{k-2}\rangle_{\BS} &= \int_{(k-1)h}^{(k-1)h+\path_i} \psi_k(x)\psi_{k-2}(x-\path_i)\dx\\
		&= \tfrac{1}{h^2}\int_0^{\path_i} (\path_i-x)x\dx = \tfrac{1}{6h^2}\path_i^3.
	\end{align*}
	For the derivatives, we obtain for suitable $k\in\{1,\ldots,n\}$
	\begin{align*}
		\langle \psi_k,\tfrac{\partial}{\partial \path_i}\transform_i(\path_i)\psi_k\rangle_{\BS} 
			&= -\int_0^1 \psi_k(x)\psi_k'(x-\path_i)\dx\\
			&= -\int_{\path_i}^h \tfrac{x}{h^2} \dx + \int_{-h}^{\path_i-h} \tfrac{x}{h^2} \dx - \int_{\path_i-h}^0 \tfrac{x}{h^2} \dx\\
			&= -\tfrac1{h^2}\left(2\path_ih-\tfrac32p_i^2\right),\\
		\langle \psi_k,\tfrac{\partial}{\partial \path_i}\transform_i(\path_i)\psi_{k+1}\rangle_{\BS} 
			&=  -\int_{kh+p_i}^{(k+1)h} \psi_k(x)\psi_k'(x-\path_i)\dx = \int_{\path_i-h}^0 \tfrac{x}{h^2}\dx\\
			&= -\tfrac{1}{2h^2}(h-\path_i)^2,\\
		\langle \psi_k,\tfrac{\partial}{\partial \path_i}\transform_i(\path_i)\psi_{k-1}\rangle_{\BS} 
			&= \int_0^{\path_i} -\tfrac{x}{h^2}\dx + \int_{\path_i}^{h} \tfrac{x}{h^2} \dx - \int_{-h}^{\path_i-h} \tfrac{x}{h^2}\dx\\
			&= -\tfrac1{h^2}\left(2\path_i^2-2h\path_i-\tfrac12(h-\path_i)^2\right),\\
		\langle \psi_k,\tfrac{\partial}{\partial \path_i}\transform_i(\path_i)\psi_{k-2}\rangle_{\BS} 
			&= \int_{(k-1)h}^{(k-1)h+p_i} \tfrac{x-(k-1)h}{h^2} \dx = \tfrac1{h^2} \int_0^{p_i} x\dx = \tfrac1{2h^2}\path_i^2,
	\end{align*}
	which concludes the proof.\qed
\end{proof}

\end{document}